\documentclass[reqno]{article}
\usepackage{amsfonts}
\usepackage{amsmath}
\usepackage{amsthm}
\usepackage{amssymb}
\usepackage{oldgerm}
\usepackage{fancyhdr}
\usepackage{fancybox}
\usepackage{mathrsfs}
\usepackage{epsfig}
\def\phi{\varphi }

\swapnumbers
\theoremstyle{plain}
\newtheorem{theorem}{Theorem}[section]
\newtheorem{corollary}[theorem]{Corollary}
\newtheorem{lemma}[theorem]{Lemma}
\newtheorem{proposition}[theorem]{Proposition}
\theoremstyle{definition}

\newtheorem{remark}[theorem]{Remark}
\newtheorem{remarks}[theorem]{Remarks}
\theoremstyle{remark}

\newtheorem{Bessel}[theorem]{Multivariate Bessel functions}
\numberwithin{equation}{section}

\begin{document}
\title{Integral representation  and sharp asymptotic results for some
  Heckman-Opdam hypergeometric functions of type BC}

\author{Margit R\"osler\\Institut f\"ur Mathematik, Universit\"at Paderborn\\
Warburger Strasse 100,
D-33098 Paderborn\\
and\\
Michael Voit\\
Fakult\"at Mathematik, Technische Universit\"at Dortmund\\
          Vogelpothsweg 87,
           D-44221 Dortmund, Germany\\
e-mail: michael.voit@math.uni-dortmund.de}
\date{\today}

\maketitle

\begin{abstract}
 The Heckman-Opdam hypergeometric functions  of type BC 
extend classical Jacobi functions in one variable and include
  the spherical functions  of non-compact  Grassmann manifolds 
 over the real, complex or quaternionic numbers.
There are various limit transitions known for such hypergeometric functions, see e.g. \cite{dJ}, \cite{RKV}. 
 In the present paper,
we use an explicit form of the Harish-Chandra integral representation as well as
an interpolated variant,
in order to obtain limit results for three continuous classes of hypergeometric functions of type BC 
which are distinguished by explicit, sharp and uniform error bounds. 
  The first limit realizes the approximation of the spherical functions of infinite dimensional 
Grassmannians of fixed rank; here  hypergeometric
  functions of type A appear as limits. The second limit is a contraction limit towards
  Bessel functions of Dunkl type.
\end{abstract}




\section{Introduction}

The theory of hypergeometric functions associated with root systems provides a
framework which generalizes the classical theory of spherical functions on
Riemannian symmetric spaces; see \cite{H},
\cite{HS} and \cite{O1} for the general theory,
as well as \cite {Sch} and \cite{NPP} for some more recent developments. Here
we consider the non-compact Grassmannians
$\mathcal G_{p,q}(\mathbb F)=G/K$ over one
of the (skew-)
fields $\mathbb F= \mathbb R, \mathbb C, \mathbb H$, 
where $G$
is one of the indefinite
orthogonal, unitary or symplectic  groups
$ SO_0(q,p),\, SU(q,p)$  or $Sp(q,p)$ with  $p>q$, and $K$ is the maximal
compact subgroup  $K=
SO(q)\times SO(p), \, S(U(q)\times U(p))$ or $Sp(q)\times Sp(p),$ respectively.
The real rank of $G/K$ is $q$, and the restricted root
system $\Delta(\frak g, \frak a)$ is of type $BC$.
Let $F_{BC}(\lambda, k;t)$ denote the Heckman-Opdam hypergeometric function
associated with the root system
\[ R = 2\cdot BC_q = \{ \pm 2e_i, \pm 4 e_i, \pm 2e_i \pm 2e_j: 1\leq i < j \leq
q\} \subset \mathbb R^q,\]
with spectral variable $\lambda \in \mathbb C^q$  and multiplicity parameter
$k$. The spherical functions of $G/K=\mathcal G_{p,q}(\mathbb F),$ which
are
$K$-biinvariant as functions on $G$, are then given by
\[ \phi_\lambda^p(a_t)=F_{BC}(i\lambda,k_p;t)
\quad\quad (t\in \mathbb R^q)\]
with $\lambda \in \mathbb C^q $ and  multiplicity
$$k_p=(d(p-q)/2, (d-1)/2, d/2)$$
corresponding to the roots $\pm 2 e_i$,  $\pm 4 e_i$ and  $2(\pm  e_i\pm
e_j)$ respectively; here $d\in\{1,2,4\}$ denotes the dimension of
  $\mathbb
R,\mathbb C,\mathbb H$ over $\mathbb R$; see \cite{R2} and Remark 2.3.
of \cite{H}. 
In \cite{R2}, the product formula for spherical functions,
\[\varphi(g) \varphi(h) = \int_k \varphi(gkh)dk \quad (g,h\in G), \]
was made explicit in such a way that it could be extended to a
product formula
for the hypergeometric function $F_{BC}$ with mulitplicity $k_p$
corresponding to arbitrary real parameters $p> 2q-1$. This led to three
continuous series of positive product formulas for
$F_{BC}$
corresponding to  $\mathbb F=\mathbb R, \mathbb C, \mathbb H$ as well as
associated
commutative, probability-preserving convolution algebras of measures
(hypergroups in the sense of \cite{J}) on the $BC_q$-Weyl chamber
\[ C_q = \{ t= (t_1, \ldots, t_q) \in \mathbb R^q: t_1 \geq \ldots \geq t_q
\geq 0\}.\]
On the other hand, the spherical functions of $G/K$
have the Harish-Chandra integral representation
\[\phi_\lambda^p(a_t) =\int_K e^{(i\lambda-\rho)(H(a_tk))} dk,
\quad \lambda
\in \mathfrak a_\mathbb C^*\cong \mathbb C^q, \]
see \cite{Hel} or \cite{GV} for the general theory and Section \ref{spherical}
for
details in our particular case.
The Harish-Chandra integral was made explicit by Sawyer \cite{Sa} for the
real Grassmannians
$\mathcal G_{p,q}(\mathbb R)$. In the present paper, we extend
Sawyer's  representation to general $\mathbb F$ and further reduce it to a form
which allows an extension  from the spherical case with
integers $p\geq 2q$ to a positive integral representation for the three
classes of hypergeometric functions $F_{BC}$ as above, with arbitrary real
parameters $p> 2q-1$, the rank $q$ being fixed. This (in part)
generalizes
the well-known integral representation of Jacobi functions, 
which are the hypergeometric functions of type $BC$ in rank one (see \cite{K1}).
We also give an analogous integral representation for the corresponding
Heckman-Opdam polynomials.

Our integral representation (Theorem \ref{int-rep}) for  the  spherical functions of 
 $\mathcal G_{p,q}(\mathbb F)$ is closely related to
those for the  spherical functions of the type $A$ symmetric spaces
 $GL(q,\mathbb F)/U(q,\mathbb F)$. In particular, we obtain immediately that for
$p\to\infty$,
the spherical functions of $\mathcal G_{p,q}(\mathbb F)$
  tend to the  spherical functions of 
 $GL(q,\mathbb F)/U(q,\mathbb F)$, a result which was proven
 recently by completely different methods and in more generality in \cite{RKV};
see also the note
\cite{K2} for the polynomial case.

As a main result of the present paper, we
 shall deduce from our explicit integral representation a result on the rate of convergence (Theorem \ref{main-p-infty}):
the convergence of the
bounded hypergeometric functions $F_{BC}$, with multiplicities depending on
$p$ as above,
 is of order $p^{-1/2}$ for $p\to\infty$, uniformly on the chamber
$C_q$ and
locally uniformly in the spectral variable. Moreover, a corresponding result
is obtained in the unbounded case. It seems that these results cannot be
obtained by the methods  of \cite{RKV}. Corresponding results
for  $q=1$, i.e., for Jacobi functions, can be found in
\cite{V2}.
 We  also mention
that our convergence results are related to further limits, e.g.,
 to limits in \cite{D} and \cite{SK} for multivariate polynomials
as well as to the convergence of (multivariable) Bessel functions of type B to
those of
type A and related results for matrix Bessel functions in \cite{RV2},
\cite{RV3}.
We point out that our convergence results with error bounds may serve as a basis
to derive central limit theorems for random walks on the Grassmannians 
  $\mathcal G_{p,q}(\mathbb F)$ when for fixed rank $q$, the time parameter of
the random walks as well as the
dimension parameter $p$ tend to infinity in a coupled way. For results in this
direction we refer to \cite{RV3},  \cite{V2}.

In generalization of the contraction principle for Riemannian symmetric spaces, Heckman-Opdam
hypergeometric functions can be approximated
for small space
variables and large spectral parameters by corresponding
Bessel functions of Dunkl type. This was first proven in  \cite{dJ} by an asymptotic analysis of the 
Cherednik system; see also \cite{RV1}. In the present paper, we shall use 
the integral representation of Theorem \ref{int-rep} in order to obtain this approximation in our series of $BC$-cases 
(which include the spherical functions on Grassmannians), again with an explicit  error estimate.
 For the case $q=1$ and 
the use  of the  error estimate in the proof of central limit theorems we refer to
\cite{V2} and references cited there.

We finally mention that 
the Harish-Chandra integral in
Proposition 5.4.1 of \cite{HS} for the $K$-spherical functions of the
 symmetric spaces $U(p,q)/(U(p)\times SU(q))$ over $\mathbb C$ may be used to
 derive an explicit integral representation for Heckman-Opdam hypergeometric
 functions of type BC for
 a different class of parameters as considered here. For such cases, 
associated convolution structures have been derived in \cite{V3}.

\medskip

The organization of this paper is as follows:  In
Section 2 we 
treat the Harish-Chandra integral representation for the
spherical functions of  $\mathcal G_{p,q}(\mathbb F)$ as well as
 for the associated three continuous series of Heckman-Opdam 
hypergeometric functions. In Section 3 we deduce the convergence of the
spherical functions of $\mathcal G_{p,q}(\mathbb
F)$ to those of
 $GL(q,\mathbb F)/U(q,\mathbb F)$ as $p\to\infty$. Section 4 is then devoted
 to precise estimates for the rate $p^{-1/2}$  of convergence. In particular,
 in order to obtain a uniform rate for $t\in C_q$, we need a technical result 
on the convex hull of Weyl group orbits of the half sum $\rho$ of roots which
will be proven
separately in an appendix (Section 6). The quantitative contraction estimates between  hypergeometric
functions of type BC and Bessel functions of type $B$ will be presented in Section 5.

\section{An integral representation for spherical functions on Grassmann
  manifolds and hypergeometric functions of type BC}\label{spherical}

In this section, we extend Sawyer's (\cite{Sa}) integral representation for
spherical functions on real Grassmannians and deduce an explicit integral
representation (Theorem \ref{int-rep}) for three continuous series for hypergeometric functions of type
$BC.$

Let $\mathbb F$ be one of the (skew-) fields $\mathbb R, \mathbb C, \mathbb H$
and $d= \text{dim}_{\mathbb R} \mathbb F \in \{1,2,4\}.$
On $\mathbb F,$ we have
the standard involution
 $x\mapsto \overline x$ and norm $|x| = (\overline x x)^{1/2}$. By
$M_{q,p}(\mathbb F)$ we denote the set of $q\times p$ matrices over $\mathbb
F$, also viewed as $\mathbb F$-linear transformations from $\mathbb F^p$ to 
$\mathbb F^q,$ which are considered as right $\mathbb F$-vector spaces. We
write $M_q(\mathbb F) = M_{q,q}(\mathbb F).$

We consider the Grassmannians $G/K=\mathcal G_{p,q}(\mathbb F)$ where $G$ is
one of the groups
$ SO_0(p,q), SU(p,q)$ or $Sp(p,q),$ and $K$
 is the  maximal  compact subgroup 
$K=SO(p)\times SO(q), \, S(U(p)\times U(q)) , \, Sp(p)\times Sp(q),$
respectively. Note that $G$ is the identity component of 
$SU(q,p;\mathbb F)$, where $U(q,p;\mathbb F)$ is the isometry group for the
quadratic form
\[ |x_1|^2 + \ldots +|x_q|^2 - |x_{q+1}|^2 -\ldots - |x_{p+q}|^2\]
on $ \mathbb F^{p+q}$.  
In the same way,
$K$ is a
subgroup of $U(q,\mathbb F)\times
U(p,\mathbb F)$ where \[U(q,\mathbb F) = \{X\in M_q(\mathbb F): X^* X = I_q\}\]
is the unitary group
over  $\mathbb F$; here $X^* = \overline X^t$ denotes the conjugate transpose.
The Lie algebra $\mathfrak g$ of $G$ consists of the  matrices 
\[ X= \begin{pmatrix} A & B\\
       B^* & D \end{pmatrix}\,\in M_{p+q}(\mathbb F)\]
      with blocks $A=-A^*\in M_q(\mathbb F)$ and $D=-D^*\in
M_p(\mathbb F)$ satisfying $\,trA + tr D=0, $ as well as $B\in M_{q,p}(\mathbb
F).$ 
 Let $\mathfrak k$ be the Lie algebra of $K$ and $\,\mathfrak g = \mathfrak k
\oplus \mathfrak
p\,$
the associated Cartan decomposition of $\mathfrak g$, with $\mathfrak p$
consisting of
the $(q,p)$-block matrices
 \[ X= \begin{pmatrix} 0 & X\\
        X^* & 0 \end{pmatrix}, \quad X\in M_{q,p}(\mathbb F).\]
        In accordance with \cite{Sa}, we use as a maximal abelian subspace
$\mathfrak a$ of $\mathfrak p$  the set of matrices
\[ H_t = \begin{pmatrix} 0_{q\times q} & \underline t & 0_{q\times(p-q)}\\
                           \underline t & 0_{q\times q} & 0 _{q\times(p-q)}\\
                            0_{(p-q)\times q} & 0_{(p-q)\times q}
&0_{(p-q)\times (p-q)} \end{pmatrix} \]       
where $\,\underline t = \text{diag}(t_1, \ldots , t_q) $ is the diagonal matrix
corresponding
to $t=(t_1, \ldots , t_q) \in \mathbb R^q.$ 
We remark that our present notions are adjusted to those of \cite{Sa}
(with $p$ and $q$ exchanged), and are slightly different from those
used in
\cite{R2}. 

The restricted root system $\Sigma=\Sigma(\mathfrak g,\mathfrak a)$ of
$\mathfrak g$
 with respect to $\mathfrak a$ consists of the non-zero linear functionals 
$\alpha\in \mathfrak a^*$ such that
\[ \mathfrak g_\alpha = 
\{ X \in \mathfrak g: [H,X ]= \alpha(H)X \,\,  \forall\, H\in \mathfrak
a\}\not=\{0\}.\]
In our case, the root system is of type $B_q$ if $\mathbb F=\mathbb R$ and of
type $BC_q$ if $\mathbb F=\mathbb C$ or $\mathbb H.$ 
The multiplicities $m_\alpha = \dim_{\mathbb R}\mathfrak g_\alpha$ can be found
e.g. in table 9 of \cite{OV}. We shall need an explicit description of the root
spaces.
For this, define $f_i\in \mathfrak a^*$ by $\, f_i(H_t) = t_i, \, i = 1, \ldots,
q.$ 
We shall write matrices from $\mathfrak g$ in $(q,q,p-q)$-block form.
 By $E_{ij}$ we denote a matrix of appropriate size 
which has entries $0$ except position $(i,j)$, where the entry is $1$. Notice
that $E_{ij}\cdot \lambda = \lambda \cdot E_{ij}$ for $\lambda \in \mathbb
F.$ The following
list of roots is easily verified by block multiplications; in the real case, it
matches Theorem 5 of \cite{Sa}.

\begin{enumerate}
 \item[\rm{(1)}] $\alpha =\pm f_i, \, \,1\leq i\leq q$. The root space
$\mathfrak g_\alpha$ is given by  $\, \mathfrak g_\alpha  = 
\{ X_{ir}^{\pm}(\lambda): \, \lambda\in \mathbb F, \, r= 1, \ldots, p-q\}\,$
with       
   \[ X_{ir}^{\pm}(\lambda)  = \begin{pmatrix} 0  &  0 & \lambda E_{ir}\\
0  & 0  & \pm \lambda E_{ir}\\
\overline\lambda E_{ri}  & \mp\overline\lambda E_{ri}  & 0
\end{pmatrix}.\]
The multiplicity of $\alpha$ is $\,m_\alpha=d(p-q).$
\item[\rm{(2)}] $\alpha = \pm(f_i-f_j), \,\, 1\leq i < j\leq q.$ In this case, 
 $\, \mathfrak g_\alpha  = 
\{ Y_{ij}^{\pm}(\lambda): \, \lambda\in \mathbb F\}\,$ with       
   \[ Y_{ij}^{\pm}(\lambda)  = \begin{pmatrix} \pm(\lambda E_{ij}-\overline
\lambda E_{ji}) &  \lambda E_{ij} +\overline \lambda E_{ji} & 0\,\\
\lambda E_{ij} + \overline \lambda E_{ji}  &   \pm(\lambda E_{ij}-\overline
\lambda E_{ji})  & 0\, \\
0 & 0 & 0\,
\end{pmatrix}.\]
The multiplicity is $\,m_\alpha=d.$
\item[\rm{(3)}]$\alpha = \pm(f_i+f_j), \,\, 1\leq i < j\leq q.$
Here  $\, \mathfrak g_\alpha  = 
\{ Z_{ij}^{\pm}(\lambda): \, \lambda\in \mathbb F\}\,$ with       
   \[ Z_{ij}^{\pm}(\lambda)  = \begin{pmatrix} \pm(\lambda
E_{ij}-\overline\lambda E_{ji}) & -\lambda E_{ij} + \overline \lambda E_{ji} &
0\,\\
-\overline \lambda E_{ji} + \lambda E_{ij} & \pm (\overline \lambda E_{ji} -
\lambda E_{ij}) & 0\, \\
0 & 0 & 0 \,
\end{pmatrix}.\]
Again, the multiplicity is $\,m_\alpha=d.$
\item[\rm{(4)}] $\alpha=\pm 2f_i\, , \,\, 1\leq i\leq q.$ This family of roots
occurs
only for $\mathbb F=\mathbb C,\mathbb H.$ The root spaces are given by
 $\mathfrak g_\alpha= \{\lambda\cdot W_i^{\pm}: \lambda\in \mathbb F,
\overline\lambda=-\lambda\}\,$ with
\[ W_i^\pm = \begin{pmatrix}
        E_{ii} & 0 & \mp E_{ii}\\
0 & 0 & 0\\
\pm E_{ii} & 0 & -E_{ii}      
\end{pmatrix}.\]
In order to obtain a unified notion, we consider $\alpha=\pm 2f_i$ also a root
if $\mathbb F=\mathbb R$, with multiplicity zero. Then $m_\alpha= d-1$ for
$\mathbb F=\mathbb R, \mathbb C, \mathbb H.$ 
\end{enumerate}
In our unified notion, $\Sigma$ is of type $BC_q$ in all cases, with the
understanding that $0$ may occur as a multiplicity on the long roots. 
As usual, we choose the positive subsystem 
\[\Sigma_+ = \{f_i, 2f_i, 1\leq i \leq q\}\cup\{ f_i\pm f_j, \,1\leq  i<j \leq
q\}\]
 Then the weighted half sum of positive roots is 
\begin{equation}\label{rho-BC}
 \rho^{BC} =  \rho^{BC}(p) =\frac{1}{2}\sum_{\alpha\in \Sigma_+}
m_\alpha\alpha \,=\,
 \sum_{i=1}^q \bigl( \frac{d}{2}(p+q+2-2i) -1 \bigr)f_i\,.
\end{equation}
Let
\[\mathfrak n = \sum_{\alpha\in \Sigma_+} \mathfrak g_\alpha\]
and $N=\exp \mathfrak n, \, A= \exp \mathfrak a.$ Then $A$ is abelian, $N$ is
nilpotent, and $G=KAN$ is an Iwasawa decomposition of $G$.
The spherical functions of $G/K$ are given by the Harish-Chandra integral
formula
\begin{equation}\label{harish1} 
\phi_\lambda^p(a_t) =\int_K e^{(i\lambda-\rho^{BC})(H(a_tk))} dk,
\quad \lambda
\in \mathfrak a_\mathbb C^*\end{equation}
where $H(g)\in A$ denotes the unique abelian part of
$g\in
G$ in the Iwasawa decomposition $G=KAN$ 
(see e.g. \cite{GV}), and 
\begin{equation}\label{a_t}
a_t= \exp(H_t) =\begin{pmatrix} \cosh\underline t & \sinh \underline t & 0 \\
          \sinh  \underline t & \cosh  \underline t & 0 \\
   0 & 0 & I_{p-q}
       \end{pmatrix}  \end{equation}
with $  \cosh\underline t = \text{diag}(\cosh t_1,
\ldots, \cosh t_q), \, \sinh \underline t = \text{diag}(\cosh t_1,
\ldots, \cosh t_q)$.

\smallskip\noindent
We shall identify $\mathfrak a_\mathbb C^*$ with $\mathbb C^q$ via
 $\, \lambda \mapsto (\lambda_1, \ldots, \lambda_q)$ for $\lambda\in \mathfrak
a_\mathbb C^*$
 given by 
$\lambda(H_t) = \sum_{r=1} ^q\lambda_r t_r\,,$ $\lambda_r\in \mathbb C.$

\medskip
In order to state a more explicit form of the Harish-Chandra integral above,
  we need some further notation.
For a square matrix $A= (a_{ij})$ over $\mathbb F$ 
 we denote by $\Delta_r(A) = \det ((a_{ij})_{1\leq i,j\leq r})\,$ 
the $r$-th principal minor of $A$. Here, for $\mathbb F= \mathbb H, $
 the determinant is understood in the sense of Dieudonn\'{e}, i.e. 
$\det(A) = (\det_{\mathbb C} (A))^{1/2}$, when $X$ is considered as a complex
matrix.

We introduce the usual power functions on the
cone 
\[ \Omega_q = \{ x\in M_q(\mathbb F): x= x^*, x \text{ strictly positive
definite}\},\]
(c.f. \cite{FK}), Chap.VII.1.): For $\lambda \in \mathbb C^q\cong \mathfrak
a_\mathbb C^*$ and $x\in \Omega_q,$ we define
\begin{equation}\label{power-function}
 \Delta_\lambda(x) = \Delta_1(x)^{\lambda_1-\lambda_2} \cdot \ldots \cdot
\Delta_{q-1}(x)^{\lambda_{q-1}-\lambda_q}\cdot
\Delta_q(x)^{\lambda_{q}}.
\end{equation}
We also define the projection matrix
$$\sigma_0 :=
\begin{pmatrix} I_q  \\ 0_{(p-q)\times q}\end{pmatrix} \in M_{p,q}(\mathbb F).$$

The following result generalizes Theorem 16 of \cite{Sa}.

\begin{theorem}\label{first-int-rep}
For the Grassmannian $\mathcal G_{p,q}(\mathbb F),$ the spherical functions
\eqref{harish1} are given by 
\[ \varphi_\lambda^{p}(a_t) = \int_{K} \Delta_{(i\lambda-\rho^{BC})/2}(x_t(k))
dk, \,\,\lambda\in \mathbb C^q  \]
where  for $\, k = \begin{pmatrix} u & 0\\
                                                         0 & v
                                                        \end{pmatrix} \in K$ 
with $\,u\in U(q,\mathbb F), v\in U(p, \mathbb F),$
  \[x_t(k): =  (\cosh\underline t \, u + \sinh \underline t \, 
\sigma_0^*v\sigma_0)^*(\cosh\underline t \, u + \sinh \underline t \, 
\sigma_0^*v\sigma_0)\in  \Omega_q\,.\]
\end{theorem}

\begin{proof}
We closely follow \cite{Sa}. Let
\[ S = \frac{1}{\sqrt{2}}\cdot\begin{pmatrix} I_q & 0_{q\times (p-q)} & J_q \\
                                I_q & 0_{q\times (p-q)} & -J_q\\
                                0_{(p-q)\times q} & \sqrt 2 I_{p-q}&
0_{(p-q)\times q}\end{pmatrix}\quad\text{with }\,
J_q=(\delta_{i,q+1-j})_{i,j}\in M_q(\mathbb F).\]
Notice that $S^* S = I_{p+q}.$ 
Using  the explicit form of the root spaces above, one
checks  that $S^*XS$
is strictly upper triangular
for each $X\in \mathfrak n.$ Thus for $n\in N,$ the matrix $S^*nS$ is upper
triangular with entries $1$ in the diagonal. Furthermore,
\[S^* \exp(H_t) S =  \text{diag}(e^{t_1}, \ldots, e^{t_q}, 1, \ldots, 1 ,
e^{-t_q}, 
\ldots, e^{-t_1})\] with $p-q$ entries $1$. 
Consider $\, g = k\,\text{exp}(H_t)n \in KAN$ and let $1\leq r \leq q$.  As in
the proof of Proposition 14 of \cite{Sa}, we calculate the principal minors
\[\Delta_r(S^*g^* g S) = \,
\Delta_r((S^*nS)^* (S^* \exp(2H_t) S) S^*nS) \,=\, e^{2(t_1 + \ldots +
t_r)}.\]
Writing $g= k\,\text{exp}(H_t)n $ in $(q,p)$-block form as
$\,\displaystyle g = \begin{pmatrix} A & B\\ C & D\end{pmatrix},\,$
the upper left $q\times q$-block of $S^*g^* g S\,$ becomes
\[ (A + B\sigma_0)^* (A+B\sigma_0) \quad \text{with } \,\,\sigma_0 =
\begin{pmatrix} I_q  \\ 0_{(p-q)\times q}\end{pmatrix} \in M_{p,q}(\mathbb F).\]
Thus
\begin{equation}\label{Sawyer} t_r = \frac{1}{2}\log\frac{\Delta_r((A +
B\sigma_0)^* (A+B\sigma_0))}{\Delta_{r-1}((A + B\sigma_0)^*
(A+B\sigma_0))},\end{equation}
with the agreement $\Delta_0:=1$. Notice that
this generalizes Proposition 14 of \cite{Sa}, and that the arguments
of $\Delta_r$ and $\Delta_{r-1}$  belong to the cone $\Omega_q$, 
because $gS$  is non-singular.

Now consider $g= a_tk$ with $k= \begin{pmatrix} u & 0 \\ 0 & v\end{pmatrix}\in
K.$
We have
\[ a_tk = \begin{pmatrix} \cosh\underline t & \sinh \underline t & 0 \\
          \sinh  \underline t & \cosh  \underline t & 0 \\
   0 & 0 & I_{p-q}
       \end{pmatrix}\cdot \begin{pmatrix} u & 0 \\ 0 & v\end{pmatrix}\, =\,
\begin{pmatrix} \cosh\underline t \, u \,& \sinh \underline t\,
\sigma_0^* v \,\\
\,\,\ast \,\, & \,\, \ast\,\,\end{pmatrix}.\]
By \eqref{Sawyer}, this gives
\[ e^ {\lambda(H(a_tk))} = \prod_{r=1}^q
\Bigl(\frac{\Delta_r(x_t(k))}{\Delta_{r-1}(x_t(k))}\Bigr)^{\lambda_r/2}\,
     =   \,\Delta_{\lambda/2}(x_t(k)),      \]
which proves the statement.
\end{proof}

For $p\geq 2q$ we may reduce the integral in Theorem \ref{first-int-rep}
by  techniques from \cite{R1}, \cite{R2}. For this, consider
the ball
\[ B_q = \{ w\in M_q(\mathbb F): w^*w < I\},\]
where $A<B$ means that $B-A$ is (strictly) positive definite, 
as well as the probability measure $m_p$ on $B_q$ given by
\begin{equation}\label{measure-mp}
 dm_p(w) = \frac{1}{\kappa_{pd/2}}\cdot\Delta(I-w^*w)^{pd/2-\gamma}dw,
\end{equation}
where
\[ \gamma:= d\bigl(q-\frac{1}{2}\bigr) +1,\]
$\Delta$ denotes the determinant on the open cone $\Omega_q$,
 $dw$ is the Lebesgue
measure on the ball $B_q$, and  
\begin{equation}\label{def-kappa} \kappa_{pd/2}= \int_{B_q}
\Delta(I-w^*w)^{pd/2-\gamma}\> dw.
\end{equation}
Notice that $m_p$ is a probability measure on  $B_q$.

\smallskip
 By $U_0(q, \mathbb F) $ we denote the  identity component of 
$U(q,\mathbb F)$. Notice that $U(q, \mathbb F) = U_0(q, \mathbb F)$ 
for $\mathbb F = \mathbb C, \mathbb H$, while $U_0(q, \mathbb R)=SO(q).$
With these notions, we obtain the following integral representation:

\begin{corollary}\label{int-rep-spherical}
 Let $p\geq 2q$ be an integer. Then the spherical functions  \eqref{harish1}
can be written as
\begin{equation}\label{intrep_BC} \varphi_\lambda^{p}(a_t) =
\int_{U_0(q,\mathbb F)\times B_q} 
\Delta_{(i\lambda-\rho^{BC})/2}(g_t(u,w))\>
dm_p(w)du \end{equation}
where $du$ denotes the normalized Haar measure on $U_0(q)$, and
\[g_t(u,w) =   u^{-1}(\cosh\underline t \, + \sinh \underline t \,
w)^*(\cosh\underline t \,  + \sinh \underline t \, w)u\,.\]
The same formula holds with the argument $g_t(u,w)$ replaced by 
\[\widetilde g_t(u,w) =   u^{-1}(\cosh\underline t \, + \sinh \underline t \,
w)(\cosh\underline t \,  + \sinh \underline t \, w)^*u.\]
\end{corollary}

\begin{proof} 
In a first step, we replace the integral over $K$ in  Theorem
\ref{first-int-rep}
by an integral over $\,U_0(q,\mathbb F) \times U(p, \mathbb F).$ This is
achieved in the same 
way as for the integral (2.5) in \cite{R2}; it is important in this context that
the argument $x_t(k)$ 
depends only on the upper left $q\times q$-block of $v$. 
Lemma 2.1 of \cite{R2} then gives the first formula with the argument 
$\, (\cosh\underline t \, u + \sinh \underline t \, 
w)^*(\cosh\underline t \, u + \sinh \underline t \, w)\,$
 instead of $g_t(u,w)$, which is then obtained by a change of variables
$w\mapsto wu$. 
 
For the proof of the second equation, notice that for  $a:=\cosh t + \sinh t
\cdot w\in M_q(\mathbb F)$,
 the matrices $a^*a$ and $aa^*$ have the same
eigenvalues with the same multiplicities. Therefore, 
 $a^*a=vaa^*v^*$ with some $v\in U(q,\mathbb F)$ for  $\mathbb F=\mathbb
R,\mathbb C$. 
In fact, this also valid for  $\mathbb H$. To check this, write $a\in M_q(
\mathbb H)$ as
$a=a_1+ja_2$ for complex matrices $a_1,a_2\in  M_q(  \mathbb C)$, and form 
$$\chi_a:=\left(\begin{array}{cc}
 a_1&a_2\\
-\bar a_2&\bar a_1
\end{array}
\right) \quad\in M_{2q}(\mathbb C).$$
The mapping $\chi:M_q(  \mathbb H)\to M_{2q}(\mathbb C)$, $a\mapsto \chi_a$,
 is a $*$-homomorphism of algebras, and $\chi_a^*\chi_a$ and  $\chi_a\chi_a^*$
 have the same eigenvalues  as $a^*a$ and $aa^*$ respectively with the doubled
multiplicities;
 see the survey \cite{Zh}. Thus, $a^*a$ and $aa^*$ have the same
eigenvalues with the same multiplicities, and hence  $a^*a=vaa^*v^*$ with some
$v\in U(q,\mathbb H)$.

 Using  $a^*a=vaa^*v^*$ for some $v\in U(q,\mathbb F)$, we see that for
 each fixed $w\in B_q$ 
$$\int_{U_0(q,\mathbb F)} \Delta_{(i\lambda-\rho)/2}(\widetilde g(t,u,w))\> du =
\int_{U_0(q,\mathbb F)}   \Delta_{(i\lambda-\rho)/2}(u^*vaa^*v^*u)\> du =
\int_{U_0(q,\mathbb F)}     \Delta_{(i\lambda-\rho)/2}( g(t,u,w))\> du.$$
This yields the second equation.
\end{proof}

We now identify $t\in C_q$ with the matrices $a_t\in G$ as above and regard
the spherical functions $\phi_\lambda^p$ above  as functions on the Weyl
chamber $C_q$. With this agreement we now extend the integral representation 
(\ref{intrep_BC}) above from integer parameters $p\ge 2q$ to arbitrary real
parameters
$p\ge 2q-1$. For this we fix $\mathbb F$ (and thus $d=1,2,4$) and define the
functions
\begin{equation}\label{def-phi}
\phi_\lambda^p(t):=  F_{BC}(i\lambda,k_p;t) \quad\quad (t\in C_q, \>
\lambda\in\mathbb C^q)
\end{equation}
 with
$$k_p=(d(p-q)/2, (d-1)/2, d/2),$$
which are  analytic in $p$ with ${\rm Re}\> p>q$. Note that for integers
$p$, the functions $\phi_\lambda^p$ are precisely the spherical functions
 \eqref{harish1}.
For the extension  of the integral representation, we shall employ
 Carleson's theorem on analytic continuation which we 
recapitulate  from \cite{Ti}, p.186:

\begin{theorem}\label{continuation} Let $f(z)$ be holomorphic in a neighbourhood
of
$\{z\in \mathbb C:{\rm Re\>} z \geq 0\}$ satisfying $f(z) =
O\bigl(e^{c|z|}\bigr)$
on $\,{\rm Re\>}  z \geq 0$ for some $c<\pi$. 
If $f(z)=0$ for all nonnegative integers $z$, then $f$ is identically zero for
${\rm Re\>}  z>0$.
\end{theorem}

We shall prove:

\begin{theorem}\label{int-rep}
Let $p\in\mathbb R$  with $p> 2q-1$. Then the
functions \eqref{def-phi} satisfy
\begin{equation}\label{intrep-main} \varphi_\lambda^{p}(t) =
\int_{B_q\times U_0(q,\mathbb
F)} 
\Delta_{(i\lambda-\rho^{BC})/2}(g_t(u,w))\>
dm_p(w)du \end{equation}
for all $\lambda\in\mathbb C^q$ and $t\in C_q\,,$
where again the argument $g_t$ may be replaced by  $\,\widetilde g_t$  as in
Corollary
\ref{int-rep-spherical}.
\end{theorem}

\begin{proof}
We first observe that both sides of  (\ref{intrep-main}) 
 are analytic in 
$p$ and $\lambda$.
In order to employ Carleson's theorem to extend  (\ref{intrep_BC})
to $p\in]2q-1,\infty[$, we need a
 suitable exponential growth bound on
  $F_{BC}$ w.r.t.~$p$  in some right half
plane. 
 Such  exponential estimates
are available only for real, nonnegative multiplicities; see Proposition 6.1 of
 \cite{O1}, \cite{Sch}, and Section 3 of \cite{RKV}. We thus  proceed in
 two steps and closely
 follow the proof of Theorem 4.1 of \cite{R2}, where a product formula is
obtained by analytic continuation. We first restrict our attention
 to a discrete set of spectral parameters $\lambda$ for which  $F_{BC}$
 is a (renormalized) Jacobi polynomial and where the growth condition is easily
checked.
Carleson's theorem then leads
 to  (\ref{intrep-main})  for this discrete set
 of  parameters $\lambda$  and all
$p\in]2q-1,\infty[$.
 In a further step we fix $p\in]2q-1,\infty[$ and   extend  (\ref{intrep-main})
 to all   $\lambda\in \mathbb C^q$.

Let us go into details. We need some notation and
facts from \cite{O1} and \cite{HS}.
For  $R=2\cdot BC_q$ with the
 set $R_+$ of positive roots,
  consider the half sum of positive roots 
\begin{equation}\label{def-rho}
\rho(k):=\frac{1}{2}\sum_{\alpha\in R_+} k(\alpha)\alpha= 
\sum_{i=1}^q (k_1 + 2k_2 +2k_3(q-i))e_i
\end{equation}
 as well
 the $c$-function
\begin{equation}\label{c_function} c(\lambda,k) := \prod_{\alpha\in R_+}
\frac{\Gamma(\langle\lambda,\alpha^\vee\rangle +
\frac{1}{2}k(\frac{\alpha}{2}))}{\Gamma (\langle\lambda,\alpha^\vee\rangle +
\frac{1}{2}k(\frac{\alpha}{2}) + k(\alpha))}\cdot \prod_{\alpha\in R_+}
\frac{\Gamma
(\langle\rho(k),\alpha^\vee\rangle + \frac{1}{2}k(\frac{\alpha}{2}) +
k(\alpha))}{\Gamma(\langle\rho(k),\alpha^\vee\rangle +
\frac{1}{2}k(\frac{\alpha}{2}))}\end{equation}
with the usual inner product on $\mathbb C^q$ and 
the conventions $\alpha^\vee:=2\alpha/\langle\alpha,\alpha\rangle$
and
 $k(\frac{\alpha}{2}) = 0$ for $\frac{\alpha}{2}\notin R$. The $c$-function
 is meromorphic on 
$\mathbb C^q\times\mathbb C^3$. We  consider the dual root system
  $R^\vee = \{\alpha^\vee: \alpha\in R\}$,
the  coroot lattice  $Q^\vee = \mathbb Z.R^\vee$, and the weight lattice
$P=\{ \lambda\in \mathbb R^q: \langle\lambda,\alpha^\vee\rangle \in \mathbb Z
\,\,\forall\,\alpha\in R\}.$ Further, denote by
$\, P_+ =\{ \lambda\in P: \langle\lambda,\alpha^\vee\rangle \geq 0
\,\,\forall\,\alpha\in R_+\}\,$  the set of dominant weights associated with
$R_+$. In our case, $P_+ = C_q\cap 2\mathbb Z^q.$
According to Eq.~(4.4.10) of \cite{HS}, we have for 
 $k\ge0$ and  $\lambda\in P_+$ the connection
\begin{equation}\label{fbc-jacobi}
 F_{BC}(\lambda +\rho(k), k;t) = \,c(\lambda+\rho(k),k)
  P_\lambda(k;t)
\end{equation}
 where the 
 $P_\lambda$ are the Heckman-Opdam Jacobi polynomials associated with  $BC_q$.
We also  consider the specific multiplicities
$k_{p}:=(d(p-q)/2, (d-1)/2, d/2) $
 and the associated half sums 
$\rho(k_p) = \rho^{BC}$ as in
 (\ref{rho-BC}). With these notations we obtain from \eqref{fbc-jacobi} and 
\eqref{def-phi} that the integral representation \eqref{intrep-main}
can be written as
\begin{equation}\label{modi-int-rep-pol}
P_\lambda(k_p;t)= \frac{1}{c(\lambda+\rho(k_p),k_p)}\cdot
\frac{1}{\kappa_{pd/2}}
\int_{B_q}\int_{U_0(q,\mathbb F)}
\Delta_{\lambda/2}( g_t(u,w))\Delta(I-w^*w)^{pd/2-\gamma}  dwdu.
\end{equation}

Exactly as in the proof of Theorem 4.1 of \cite{R2}, it is now checked that both
sides of \eqref{modi-int-rep-pol} are, as functions of $p$,  of polynomial
growth in the half-plane $\{ p\in \mathbb C: \text{Re} (pd/2) >\gamma-1\};$ we
omit the details.
We may therefore apply Carleson's
 theorem to \eqref{modi-int-rep-pol}, and this proves \eqref{intrep-main} for
$p$
 with ${\rm Re}(pd/2)>\gamma-1$ and all spectral parameters of the form
 $-i(\lambda+\rho(k_p))$
with $\lambda\in P_+$.

We next fix $p\in\mathbb R$ with $p>2q-1$ (in which case $k_p$ is
nonnegative) and extend \eqref{intrep-main} with respect to the spectral
parameter $\lambda$. According to
Proposition 6.1 of \cite{O1}, 
$$|F_{BC}(\lambda,k_p;t)|\le
 |W|^{1/2}e^{\max_{w\in W} {\rm Re}\> \langle w\lambda, t\rangle}$$ where $W$
 is the Weyl group of $BC_q$. Let $C_q^0$
 be the interior of $C_q$ and 
$H^\prime:=\{\lambda\in\mathbb C^q:\>  {\rm Re}\lambda\in C_q^0\}$. Then 
$${\rm Re}\langle w\lambda, t\rangle \le {\rm Re}\langle \lambda, t\rangle
\quad\quad\text{for}\quad
\lambda\in H^\prime, \> t\in C_q,\> w\in W.$$
Now fix $t\in C_q$ and $p$ as above, 
and choose a vector $a\in C_q^0$ sufficiently large.
Then \eqref{intrep-main} for the spectral parameter $\lambda+\rho(k_p)$
is equvalent to
$$e^{-\langle \lambda, a+t\rangle}
\varphi_{-i(\lambda+\rho(k_p))}^{p}(t) = \int_{B_q\times U_0(q,\mathbb F)} 
e^{-\langle \lambda, a+t\rangle}\cdot \Delta_{\lambda/2}(g_t(u,w))\>
dm_p(w)du. $$
The left hand side remains bounded
for
 $\lambda\in H^\prime$.
 Moreover, for  $a\in C_q^0$ sufficiently large,
$$\sup_{(u,w)\in U_0(q,\mathbb F)\times B_q;\lambda\in H^\prime }
 |e^{-\langle \lambda, a+t\rangle}\cdot \Delta_{\lambda/2}(g_t(u,w))|<\infty,$$
which proves that also the right hand side  remains bounded for
 $\lambda\in H^\prime$. By a $q$-fold application of Carleson's theorem
we thus may extend the preceding equation from $\lambda\in P_+$ to
$\lambda\in H^\prime$. 
A classical analytic continuation now finishes the proof.
\end{proof}

The above proof reveals in particular the following  integral representation for  Heckman-Opdam polynomials
of type $BC$:

\begin{corollary} Let $k_p = (d(p-q)/2, (d-1)/2, d/2)$ with $p\in \mathbb R, \, p> 2q-1.$
 Then the  Heckman-Opdam polynomials of type $BC_q$ with multiplicity 
$k_p$ have the integral representation 
\[ P_\lambda(k_p;t)= \frac{1}{c(\lambda+\rho(k_p),k_p)}
\int_{B_q\times U_0(q,\mathbb F)}
\Delta_{\lambda/2}( g_t(u,w))dm_p(w)du \quad \text{ for } t\in \mathbb C^q.\]
Here $\lambda \in P_+ = C_q\cap 2\mathbb Z^q$ and 
\[g_t(u,w) = u^{-1}(\cosh\underline t \, + w^*\sinh \underline t)(\cosh\underline t \, 
 + \sinh \underline t \, w)u\,.\]
\end{corollary}

\begin{remark}\label{degen-case-2q-1} For the limit case $p=2q-1$, a degenerate
version of the
  integral representation  
(\ref{intrep-main}) is available. For this we follow
Section 3 of \cite{R1}.

We fix the dimension $q$ and consider the matrix ball 
$B_q:=\{w\in M_q(\mathbb F):\> w^*w < I_q\}$ as above as well as the ball
$B:=\{y\in\mathbb F^q:\> \|y\|_2= (\sum_{j=1}^{q}\overline y_j y_j)^{1/2} <1\}$
and the sphere
$S:=\{y\in\mathbb F^q:\> \|y\|_2=1\}$.
By Lemma 3.7 and Corollary 3.8 of \cite{R1}, the mapping 
\begin{equation}\label{trafo-P1}
 P(y_1, \ldots, y_q):= \begin{pmatrix}y_1\\y_2(I_q-y_1^*y_1)^{1/2}\\ 
\vdots\\
  y_q(I_q-y_{q-1}^*y_{q-1})^{1/2}\cdots (I_q-y_{1}^*y_{1})^{1/2}\end{pmatrix},
\quad y_1, \ldots, y_q\in B
\end{equation}
establishes a diffeomorphism  $\,P: B^q \to B_q$. The image of the measure
$dm_p(w)$ 
under  $P^{-1}$ is given by
\begin{equation}\label{image_meas}
\frac{1}{\kappa_{pd/2}}\prod_{j=1}^{q}
(1-\|y_j\|_2^2)^ { d(p-q-j+1)/2-1 }
dy_1\ldots dy_q\,.\end{equation}
Thus for $p>2q-1$, the integral representation
(\ref{intrep-main}) may be rewritten as
\begin{equation}\label{intrep-main-mod}
\phi_\lambda^p(t)=
\frac{1}{\kappa_{pd/2}}\int_{B^q}\int_{U_0(q,\mathbb F)} 
\Delta_{(i\lambda-\rho^{BC})/2}( g_t(u,P(y)))
\cdot\prod_{j=1}^{q}(1-\|y_j\|_2^2)^{d(p-q-j+1)/2-1}dy_1\ldots
dy_q dw
\end{equation}
where $dy_1,\ldots,dy_q$ means integration w.r.t.~the Lebesgue
measure on $\mathbb F^q$.
 Moreover,  for $p\downarrow 2q-1$,  (\ref{intrep-main-mod}) 
and continuity lead to the following degenerated
 product formula:
\begin{align}\label{intrep-main-degen}
\phi_\lambda^{2q-1}(t)=
\frac{1}{\kappa_{(2q-1)d/2}}\int_{B^{q-1}}\int_S\int_{U_0(q,\mathbb F)}
\Delta_{(i\lambda-\rho^{BC})/2}( g_t(u,P(y)))\> 
\cdot &\notag\\
 \cdot\prod_{j=1}^{q-1}(1-\|y_j\|_2^2)^{d(q-j)/2-1}dy_1\ldots
dy_{q-1}\> d\sigma(y_q)\> dw
\end{align}
where $\sigma\in M^1(S)$ is the uniform distribution on the sphere 
$S$
 and
$$\kappa_{(2q-1)d/2}=\int_{B^{q-1}}\int_S\prod_{j=1}^{q-1}(1-\|y_j\|_2^2)^{
d(q-j)/2-1 }
dy_1\ldots
dy_{q-1}\> d\sigma(y_q).$$
Notice that the $\phi_\lambda^{2q-1}$ are the spherical functions of the
Grassmannian  $\mathcal G_{2q-1,q}(\mathbb F)$.
\end{remark}

\section{The connection with spherical functions of type $A_{q-1}$}

We shall compare the spherical functions  of
the Grassmannians
$\mathcal G_{p,q}(\mathbb F)$  with the
spherical
functions of the symmetric space $\mathcal P_q(\mathbb F) =
G/K$ with
$G=GL(q, \mathbb F), K = U(q, \mathbb F).$ It is well-known that $G$ has the
Iwasawa decomposition
$G = K A N$ where $A = \exp \mathfrak a, \,\,\mathfrak a = \{ H_t = \underline
t, \, t=(t_1, \ldots, t_q) \in \mathbb R^ q\}$
and $N$ is the unipotent group consisting of all
 upper trangular matrices with entries $1$ in the diagonal. 
The restricted root system
$\Delta(\mathfrak g, \mathfrak a)$ is of type $A_{q-1},$ with a positive
subsystem  given by
\[\Delta_+ = \{ f_i - f_j: 1 \leq i<j\leq q\}.\]
Here the multiplicity is $m_\alpha = d\,$ for all $\alpha \in \Delta_+$ and the
weighted half-sum
of positive roots is
\[ \rho^A = \sum_{i=1}^q \frac{d}{2}(q+1-2i)f_i \,.\]
Again, $\mathfrak a_\mathbb C^*$ may be identified with $\mathbb C^q$ via $\,
\lambda \mapsto (\lambda_1, \ldots, \lambda_q)$ for $\lambda\in \mathfrak
a_\mathbb C^*$ given by 
$\lambda(H_t) = \sum_{r=1} ^q\lambda_r t_r\,,$ $\lambda_r\in \mathbb C.$
We briefly recall the further well-known calculation, which is 
 similar to the Grassmannian case:
For $g= k\exp(H_t)n\in KAN$ one obtains 
$\,\Delta_r(g^*g) = e^{2(t_1 + \ldots + t_r)}\,$
and thus 
\[ t_r = \frac{1}{2} \log \frac{\Delta_r(g^*g)}{\Delta_{r-1}(g^*g)}
\quad\quad (r=1,\ldots,q)
.\]
If $g = a_t k$ with $a_t = \exp(H_t) = e^{\underline t}\,$ and $k\in K$, then
$g^*g = k^{-1}e^{2\underline t}k$. 
The spherical functions of $G/K = \mathcal P_q(\mathbb F)$ 
are  given by
\begin{equation}\label{def-spherical_A} 
 \varphi_\lambda^A(e^{\underline t}) = \int_K e^{(i\lambda-\rho^A)(H(a_tk))}
 dk,
\quad \lambda \in \mathbb C^q.
\end{equation}
The above considerations lead to the known integral representation
\begin{equation}\label{spherical_A}
 \varphi_\lambda^A(e^{\underline t}) =\,
\int_{U(q,\mathbb F)} \Delta_{(i\lambda-\rho^A)/2}\bigl(u^{-1}e^{2\underline
t}\,u\bigr) du \, = 
\,\int_{U_0(q,\mathbb F)}
\Delta_{(i\lambda-\rho^A)/2}\bigl(u^{-1}e^{2\underline t}\,u\bigr) du.  
 \end{equation}
We also remark that the functions $\varphi_\lambda^A$ can be written in terms
of the Heckman-Opdam hypergeometric function $F_A$ associated with the root
system
$\,2A_{q-1} = \{\pm2(e_i-e_j): 1\leq i < j\leq q\}$, as follows:
 \begin{equation}\label{spherical_A-F_A}
\varphi_\lambda^A(e^{\underline t}) = e^{\langle t-\pi(t),\lambda\rangle}\cdot 
F_A(\pi(\lambda), d/2;\pi(t))
\quad\quad (\lambda\in\mathbb C^q,\> t\in\mathbb R^q).
\end{equation}
Here $\pi$ denotes the orthogonal projection from $\mathbb R^q$ onto
 $\mathbb R^q_0:=\{t\in \mathbb R^q: t_1+\ldots+t_q=0\}$;
see  Eq.~(6.7) of \cite{RKV}, and note our rescaling of the root system by the factor $2$. 

\smallskip
We compare \eqref{spherical_A} with the integral 
  (\ref{intrep_BC}) for the spherical functions of $\mathcal G_{p,q}(\mathbb
F)$ and, more generally, with representation \eqref{intrep-main} for 
the hypergeometric functions $\varphi_{\lambda-i\rho^{BC}}^{p}$.
As for $p\to\infty$ the probability measures $m_p$ on $B_q$ tend weakly to the
point measure  at the zero matrix, we obtain:

\begin{corollary}\label{first-limit}
 The spherical functions of $\mathcal G_{p,q}(\mathbb F),$ and more generally, the hypergeometric functions
$\varphi_{\lambda-i\rho^{BC}}^{p}$ with $ p\in \mathbb R, \, p>2q-1$  are related to the spherical functions of $\mathcal
P_q(\mathbb F)$ by
 \[\lim_{p\to \infty} \varphi_{\lambda-i\rho^{BC}}^{p}(t) = \,
\varphi_{\lambda-i\rho^A}^A(\cosh \underline t) \quad (t\in \mathbb R^ q).\]

\end{corollary}

This result was already obtained in  Corollary 6.1 of \cite{RKV}
 by completely different methods, namely  as a special case of a 
general limit transition for hypergeometric functions
of type BC. However, the approach in \cite{RKV} seems not
suitable to gain information on the rate of convergence.
In the following section, we study the  integral
representations \eqref{spherical_A} and (\ref{intrep_BC}) 
(or \eqref{intrep-main} for continuous $p$) 
in order to derive precise estimates on the rate of convergence.

\section{ The rate of convergence for  $p\to\infty$}

The main result of this section is Theorem \ref{main-p-infty}. It sharpens the
qualitative limit of Corollary \ref{first-limit} for the Heckman-Opdam
hypergeometric functions $\varphi_\lambda^p$
by a precise estimate of the approximation error. 
Again, $p> 2q-1$ varies and
the rank $q$ as well as the dimension $d=1,2,4$ of $\mathbb F$ are fixed. 
For convenience, we consider the type $A$ spherical functions
$\varphi_{\lambda}^A$  as functions on $\mathbb R^q$ and study
\begin{equation}\label{psi_int}\psi_\lambda(t):=\varphi_{\lambda}^A(\cosh \underline t)=
\int_{U_0(q,\mathbb F)}
\Delta_{(i\lambda-\rho^A)/2}\bigl(u^{-1}\cosh^{2}\!\underline t u\bigr)
du.\end{equation}
We write
\begin{align*} &\|\lambda\|_1:=|\lambda_1|+\ldots+|\lambda_q|\quad \text{for }
\lambda\in\mathbb C^q;\\
 &\tilde t:=\,\min(t_1,1)\ge0 \,\,\text{ for }\,
t=(t_1,\ldots,t_q)\in C_q\,.\end{align*}
The action of the Weyl group $W$ of type $BC_q$ extends in a natural way to
$\mathbb C^q$. We write
\[\rho:=\rho^{BC}(p)\] for the half sum \eqref{rho-BC}. Moreover, $co(W.\rho)\subset \mathbb R^q$
 denotes the convex hull of the $W$-orbit of $\rho$.

\smallskip
Let us recacpitulate the following known properties of  $\phi_\lambda^p$:

\begin{lemma}\label{bounded-est-BC}
\begin{enumerate}\itemsep=-1pt
\item[\rm{(1)}] For all $t\in C_q$,  $\lambda\in\mathbb C^q$, and $p\in\mathbb
  R$ with $p\geq q$,
$$  \Bigl|\phi_{\lambda-i\rho}^p(t)\Bigr|\le
 e^{\max_{w\in W}{\rm  Im}\> \langle w\lambda, t\rangle}.$$
\item[\rm{(2)}] $\phi_{\lambda}^p$ is bounded  if and only if 
$\,{\rm Im}\>\lambda\in co(W.\rho)$. Moreover, in this case
$\|\phi_{\lambda}^p\|_\infty=1$.
\item[\rm{(3)}] If $\lambda$ is purely imaginary, then $\phi_{\lambda}^p$ is real-valued and strictly positive on $C^q$. 
\end{enumerate}
\end{lemma}

\begin{proof} (1) follows from Corollary 3.4 of \cite{RKV}. For part (2) we refer to Theorem 5.4
  of \cite{R2} and Theorem 4.2 of  \cite{NPP} (the proof of the only-if-part in
  \cite{R2} contains a gap). Part (3) follows from Lemma 3.1. of \cite{Sch}.
\end{proof}

Notice that by Corollary \ref{first-limit}, the same estimates as in Lemma
 \ref{bounded-est-BC}
hold for the function $\psi_{\lambda-i\rho^A}(t).$
The following theorem is the main result of this section:

\begin{theorem}\label{main-p-infty}
There  exists a universal constant $C=C(\mathbb F, q)$ as follows:
\begin{enumerate}
\item[\rm{(1)}] For all $p>2q-1$, $t\in C_q$ and $\lambda\in\mathbb C^q$,
$$  \Bigl|\phi_{\lambda-i\rho}^p(t)- \psi_{\lambda-i\rho^A}(t)\Bigr|\le
C\cdot \frac{\|\lambda\|_1\cdot\tilde t}{p^{1/2}}\cdot
 e^{\max_{w\in W}{\rm  Im}\> \langle w\lambda, t\rangle}.$$
\item[\rm{(2)}]
 Let $p>2q-1$, $t\in C_q$, and $\lambda\in\mathbb C^q$ such that $\,\text{Im}\,\lambda-\rho\,$ is contained in  $co(W.\rho)$, i.e., 
$\phi_{\lambda-i\rho}^p$ is bounded on $C_q$. Then
$$  \Bigl|\phi_{\lambda-i\rho}^p(t)- \psi_{\lambda-i\rho^A}(t)\Bigr|\le
C\cdot \frac{\|\lambda\|_1\cdot\tilde t}{p^{1/2}}.$$
In particular, for these spectral parameters $\lambda$ the order of
convergence is uniform of order $p^{-1/2}$ in $t\in C_q$.
\end{enumerate}
\end{theorem}

 We briefly discuss this result in the rank-one case $q=1$.
Here the Heckman-Opdam functions $\phi_\lambda^p$
 are   Jacobi functions $\phi_\lambda^{(\alpha,\beta)}$
 as studied in Koornwinder \cite{K1}. More precisely,
$$\phi_\lambda^p(t)=\phi_\lambda^{(\alpha,\beta)}(t) 
\quad\quad\text{with }\,
\alpha=dp/2, \> \beta=d/2-1, \> d=1,2,4$$
and $\rho=\alpha+\beta+1=d(p+1)/2$.
Furthermore,
 \[\psi_\lambda(t)=e^{i\lambda\cdot\ln(\cosh t)} = (\cosh t)^{i\lambda}\] independently of 
$d$, and
 $\rho^A=0$.
Thus,  Theorem \ref{main-p-infty} implies for $q=1$ the following

\begin{corollary}\label{main-p-1-infty}
There exists a constant $C>0$ as follows:
\begin{enumerate}
\item[\rm{(1)}] For $\beta=-1/2,0,1$, all $t\in [0,\infty[$,  $\alpha>0$,
 and $\lambda\in\mathbb C,$
$$\Bigl|\phi_{\lambda-i\rho}^{(\alpha,\beta)}(t)-  (\cosh t)^{i\lambda}\Bigr|\le \,C\cdot
 \frac{|\lambda|\, \min(t,1)}{\sqrt\alpha}\cdot e^{|{\rm  Im}\> \lambda|\cdot t}.$$
\item[\rm{(2)}]
 Let  $\beta=-1/2,0,1$, $t\in [0,\infty[$,   $\alpha>0$, and
 $\lambda\in\mathbb C$ with  $\,\text{Im}\,\lambda\in[0,2\rho]$.
 Then
$$\Bigl|\phi_{\lambda-i\rho}^{(\alpha,\beta)}(t)-  (\cosh t)^{i\lambda}\Bigr|\leq
 C\cdot
 \frac{|\lambda|\, \min(t,1)}{\sqrt\alpha}.$$
\end{enumerate}
\end{corollary}

\begin{remarks}
\begin{enumerate}
\item[\rm{(1)}] For ${\rm Im}\>\lambda=0\,$ and all $\beta\ge -1/2$,
 Corollary \ref{main-p-1-infty}(2) was proven
  in \cite{V2}. The proof there  relies 
on  the well-known integral
 representation for the  Jacobi functions 
for  $\alpha\ge\beta\ge -1/2$ in \cite{K1} and is similar to
that given here.
  Corollary \ref{main-p-1-infty} (2) for
${\rm Im}\>\lambda=0$ is used in  \cite{V2}
to derive a central limit theorem for the hyperbolic distances of radial random
walks on
hyperbolic spaces from their starting point when the number of time steps as
well as the dimensions of the hyperbolic spaces tend to infinity. Similar 
results can be derived from Theorem  \ref{main-p-infty} for $q\ge 2$.

\item[\rm{(2)}]  Corollary \ref{main-p-1-infty}
 corresponds to the convergence of the well-known one-dimensional Jacobi
 convolutions
 $*_{(\alpha,\beta)}$ to
 a semigroup convolution on $[0,\infty[$  in \cite{V1} where the
 multiplicative functions of the limit semigroup are precisely the functions 
$t\mapsto (\cosh t)^{i\lambda}\,$; i.e., the convergence of 
the convolution structures $*_{(\alpha,\beta)}$ for $\alpha\to\infty$ corresponds to the convergence
of the
 multiplicative functions. The same picture appears for $q>1$; see \cite{R2} for
the explicit
convolution and
\cite{RKV} for the corresponding limit transition.
 In \cite{K2}, a corresponding result for polynomials was derived.

\item[\rm{(3)}]  There are similar limit results to those of Theorem  \ref{main-p-infty} for 
 Dunkl-type Bessel functions of types A and B,  and for
Bessel functions on matrix cones   with applications 
in probability; see \cite{RV2}, \cite{RV3}. 
\end{enumerate}
\end{remarks}

We  now turn to the proof of Theorem
\ref{main-p-infty}.
In fact, our main result is essentially a consequence of Lemma
\ref{bounded-est-BC}
 and the following  technical variant of Theorem
\ref{main-p-infty}:

\begin{theorem}\label{absch-q-p-infty}
For each $n\in\mathbb N$ there  is a  constant $C=C(\mathbb F, q,n)$
 such that for all $p>2q-1$,
 $t\in C_q$ and
$\lambda\in\mathbb C^q$,  
\begin{equation}\label{absch1}
\Bigl|\phi_{\lambda-i\rho}^p(t)- \psi_{\lambda-i\rho^A}(t)\Bigr|\le
 C\cdot \Bigl(\phi^p_{\!\frac{2n}{2n-1}i\text{Im}
   \lambda-i\rho}(t)^{\frac{2n-1}{2n}}
 +\psi_{\frac{2n}{2n-1}i\text{Im} \lambda-i\rho^A}(t)^{ \frac{2n-1}{2n}}\Bigr)
 \frac{\|\lambda\|_1\cdot\tilde t}{p^{1/2}}\,.
\end{equation}
\end{theorem}
Notice that the functions $\phi, \psi\,$  on the right side take positive values by Lemma \ref{bounded-est-BC}. 
In fact, Theorem \ref{main-p-infty}(1) follows immediately from Lemma 
\ref{bounded-est-BC}(1) and Theorem \ref{absch-q-p-infty} with $n=1$.
For the proof of Theorem \ref{main-p-infty}(2), consider $\lambda\in\mathbb
C^q$ with $\,\text{Im}\lambda-\rho\in co(W.\rho)$.
As $\phi_{\lambda}^p$ is $W$-invariant in the spectral variable $\lambda$  and as the mapping $\lambda\mapsto -\lambda$
is an element of $W$, we may assume without loss of generality that
$\,\text{Im}\lambda -\rho\in -C_q\,.$
Now choose $\epsilon_0=\epsilon_0(q)>0$ according to 
the following  Lemma \ref{conv-hull}, and choose $n\in\mathbb N$ such that
 $\epsilon:=(2n-1)^{-1}\le\epsilon_0$. Lemma \ref{conv-hull} for $\,y:=\text{
Im}\lambda-\rho$
thus implies that
$$\frac{2n}{2n-1}{\rm Im}\> \lambda-\rho=(1+\epsilon)\text{Im} \lambda-\rho=
(1+\epsilon)y+\epsilon \rho\in co(W.\rho).$$
This fact,
 Lemma \ref{bounded-est-BC}(2), and   Theorem \ref{absch-q-p-infty} then lead
to 
 Theorem \ref{main-p-infty}(2) as claimed.

\begin{lemma}\label{conv-hull}
For each dimension $q$ there exists a constant $\epsilon_0=\epsilon_0(q)>0$
such that for all $0<\epsilon\le\epsilon_0$, all $\rho$ in the
interior of $C_q$, and all $y\in  co(W.\rho)\cap(-C_q)$,
$$(1+\epsilon)y+\epsilon\rho\in  co(W.\rho).$$
\end{lemma}

The proof of this lemma will be postponed to an appendix at the end of this
paper.
We here only mention that for $q=1,2$ the lemma
 can be easily checked with $\epsilon_0=1$ at hand of a
picture, but for $q\geq 3$ the situation is more complicated, and the lemma is then
no longer true with  $\epsilon_0=1$.

\medskip

We now turn to the technical proof of Theorem  \ref{absch-q-p-infty}.
 We decompose it into several steps.
We first recall the integral representation (\ref{intrep-main}),  
\begin{equation}\label{wdh-int-rep1}
\phi_{\lambda-i\rho}^p(t) =\int_{B_q} \int_{U_0(q,\mathbb F)}
 \Delta_{i\lambda/2}(\widetilde g_t(u,w))\> dm_p(w)du
\end{equation}
with the probability measure $dm_p$ as in Section 2 and
\begin{equation}\label{bwtilde}
\widetilde g_t(u,w)=u^*(\cosh \underline t + \sinh \underline t\, w)(\cosh \underline t + \sinh\underline t\, w)^*u. 
\end{equation}
In order to analyze the principal minors $\Delta_1,\ldots,\Delta_q$ appearing
in the definition of the power function $\Delta_{i\lambda/2}$, we
use  the singular values 
$\sigma_1(a)\ge\sigma_2(a)\ge\ldots\ge \sigma_q(a)$ of a matrix $a\in M_q$
ordered by size, i.e., the
square roots of the eigenvalues of $a^*a$. We  need the following known
estimates  for singular values:

\begin{lemma}\label{lemma-eins1}
For all matrices $a_1,a_2\in M_q(\mathbb F)$ and $i=1,\ldots, q$,
$$|\sigma_i(a_1+a_2)-\sigma_i(a_1)|\le\sigma_1(a_2)
 \quad\quad\text{and} \quad\quad 
\sigma_i(a_1\cdot a_2)\le \sigma_i(a_1)\sigma_1(a_2).$$
\end{lemma}

\begin{proof} For $\mathbb F=\mathbb R,\mathbb C$ we refer to Theorem 3.3.16 of
  \cite{HJ}. The case $\mathbb F=\mathbb H$ can be reduced to $\mathbb F=\mathbb
C$
 by the same arguments as  in the second part of the proof of Corollary
\ref{int-rep-spherical}.
\end{proof}

\begin{lemma}\label{lemma-eins2} For $t\in C_q$, $w\in B_q$, $u\in
  U_0(q,\mathbb F)$ and $r=1,\ldots,q$,
\[\frac{\Delta_r( \widetilde g_t(u,w))}{\Delta_r( \widetilde g_t(u,0))}\in
 \big[(1-  \widetilde t\,\sigma_1(w))^{2r}, (1+  \widetilde t\,\sigma_1(w))^{2r}\big], \quad \text{with } \,
\widetilde t:=\min(t_1,1). \]
\end{lemma}

\begin{proof}
We write the matrix $ \widetilde g_t(u,w)$ as
\begin{equation}\label{bwtilde1}
\widetilde g_t(u,w)=
 b(I+\widetilde w)(I+\widetilde w^*)b^*
\end{equation}
with
$$b:=u^* \cosh\underline t,   \quad
 \widetilde w:=( \cosh \underline t)^{-1}\sinh \underline t\cdot w\, = \tanh \underline t\cdot w$$
The inequalities of Lemma \ref{lemma-eins1} imply for  $i=1,\ldots, q$ that
\begin{align}
|1-\sigma_i(I+\widetilde w)|&= |\sigma_i(I)-\sigma_i(I+\widetilde w)|\le \sigma_1(\widetilde
w)=
\sigma_1(\tanh\underline t\cdot w)
\notag\\
&\le \sigma_1(\tanh \underline t\,)\cdot \sigma_1(w) .\notag
\end{align}
As $\,0\le \tanh x \le \min(x,1)$ for
$x\ge0$ and $x\mapsto\tanh x$ is increasing, we conclude that
$$ \sigma_1(\tanh\underline  t\,)\le \,\min(t_1,1) =\widetilde t$$
and thus
\begin{equation}
|1-\sigma_i(I+\widetilde w)|\le\, \widetilde t\cdot \sigma_1(w) \,\in [0,1].
\end{equation}
This implies for $i=1,\ldots,q\,$  that
\begin{equation}\label{sing-value}
(1- \widetilde t\,\sigma_1(w))^2 \leq \sigma_i(I+\widetilde w)^2\leq (1+ \widetilde
t\,\sigma_1(w))^2.
\end{equation}
This leads to the matrix inequality
$$(1- \widetilde t\sigma_1(w))^2I\le
(I+\widetilde w)(I+\widetilde w^*)\le(1+ \widetilde t\sigma_1(w))^2I,$$
and thus
$$(1-  \widetilde t\sigma_1(w))^2\,bb^*\le b(I+\widetilde w)(I+\widetilde w^*)b^*
\le(1+  \widetilde t\sigma_1(w))^2\,bb^*.$$
As for Hermitian  matrices $a,b$ with $0\leq a\le b$ the determinants
satisfy
 $\,0\le \Delta(a)\le \Delta(b)$, we finally obtain 
\begin{equation}\label{delta-r-ab}
\Delta_r( b(I+\tilde w)(I+\tilde w^*)b^*)\in
 \big[(1-  \tilde t\sigma_1(w))^{2r}\Delta_r(bb^*),
 (1+  \tilde t\sigma_1(w))^{2r}\Delta_r(bb^*) \big]
\end{equation}
as claimed.
\end{proof}

For the next step in the proof of Theorem 4.5 we use the integral representation \eqref{psi_int},
\begin{equation}\label{wdh-int-rep2}
\psi_{\lambda-i\rho^A}(t)=\int_{U_0(q,\mathbb F)} 
 \Delta_{i\lambda/2}(u^{-1}(\cosh\underline  t)^2u)du\,=\,
 \int_{B_q}\int_{U_0(q,\mathbb F)}  \Delta_{i\lambda/2}(\widetilde g_t(u,0))dm_p(w)du.
\end{equation}
 Using Lemma \ref{lemma-eins2}, we
estimate
the difference of the integrands in (\ref{wdh-int-rep1}) and
(\ref{wdh-int-rep2}).
We shall obtain the following result.

\begin{lemma}\label{lemma-eins3} Let $t\in\mathbb R^q$ and $\lambda\in\mathbb C^q$. Then for all
  $n\in\mathbb N$,
\begin{align}
\bigl| \phi_{\lambda-i\rho}^p(t)-\psi_{\lambda-i\rho^A}(t)\bigr|\,\leq&\,\,
 8q\|\lambda\|_1\,\widetilde t \cdot\Bigl(
\frac{1}{\kappa_{pd/2}}\int_{B_q}\sigma_1(w)^{2n}
 \Delta(I-w^*w)^{pd/2-\gamma -2n}\, dw\Bigr)^{1/2n}\notag\\
&\quad\cdot\Bigl(\phi^p_{\!\frac{2n}{2n-1}i\text{Im}
\lambda-i\rho}(t)^{\frac{2n-1}{2n}}
 +\psi_{\frac{2n}{2n-1}i\text{Im} \lambda-i\rho^A}(t)^{\frac{2n-1}{2n}}\Bigr)
\notag\end{align}
\end{lemma}

\begin{proof}
We write the difference 
$$D:=\bigl|\Delta_{i\lambda/2}(\widetilde g_t(u,w))-\Delta_{i\lambda/2}(\widetilde
  g_t(u,0))\bigr|$$
 of the integrands in (\ref{wdh-int-rep1}), (\ref{wdh-int-rep2})
as
$\,D=|e^\alpha-e^\beta|\,$ with
$$\alpha:=\alpha(t,\lambda,u,w)= \frac{i}{2}\sum_{r=1}^q
(\lambda_r-\lambda_{r+1})
 \cdot\ln \Delta_r(\widetilde g_t(u,w))$$
and
$$\beta:=\beta(t,\lambda,u)= \frac{i}{2}\sum_{r=1}^q (\lambda_r-\lambda_{r+1})
 \cdot\ln \Delta_r(\widetilde g_t(u,0))$$
with the agreement $\lambda_{q+1}=0$.
We further write the functions $\alpha,\beta$ as $\alpha=\alpha_1+i\alpha_2$
 and $\beta=\beta_1+i\beta_2$ with 
$\alpha_1,\alpha_2,\beta_1,\beta_2\in \mathbb R$. By elementary calculus, we
obtain
\begin{align}\label{exp-absch-alpha}
|e^\alpha-e^\beta|&= \,|e^{\alpha_1+i\alpha_2}-e^{\beta_1+i\beta_2}|\le 
|e^{i\alpha_2}|\cdot|e^{\alpha_1}-e^{\beta_1}|+ e^{\beta_1}\cdot| e^{i\alpha_2}-
e^{i\beta_2}|
\notag\\ &\le |e^{\alpha_1}-e^{\beta_1}|+\sqrt 2 \cdot
e^{\beta_1}|\alpha_2-\beta_2|
\notag\\ &\le |\alpha_1-\beta_1|\cdot(e^{\alpha_1}+e^{\beta_1}) 
+\sqrt 2  (e^{\alpha_1}+e^{\beta_1})|\alpha_2-\beta_2|
\notag\\ &\le \,2\cdot|\alpha-\beta|\cdot(e^{\alpha_1}+e^{\beta_1}).
\end{align}
We have 
$$ |\alpha-\beta|\le \|\lambda\|_1 \cdot
\max_{r=1,\ldots,q}\bigl|\ln  \Delta_r(\widetilde g(t,u,w))-
\ln  \Delta_r(\widetilde g(t,u,0))\bigr|.$$
Hence we obtain from Lemma  \ref{lemma-eins2}, together with  the elementary inequality
\begin{equation}\label{element-ungl-ln}
|\ln(1+z)|\le \frac{|z|}{1-|z|} \quad\text{for}\quad |z|<1
\end{equation}
and with $\widetilde t\in[0,1]$ that
$$ |\alpha-\beta|\le \|\lambda\|_1 \cdot 2q
\cdot\frac{\widetilde t\,\sigma_1(w)}{1-\tilde t\, \sigma_1(w)}
\le \|\lambda\|_1 \cdot 2q\,
\widetilde t\cdot\frac{\sigma_1(w)}{1- \sigma_1(w)}.$$
Furthermore, as $1\ge\sigma_1(w)\ge\ldots\ge\sigma_q(w)\ge0$ for $w\in B_q$,
we have
\begin{equation}\label{est-sigma1-delta}
\frac{1}{1-\sigma_1(w)}\le \frac{2}{1-\sigma_1(w)^2}\le \,2 \prod_{r=1}^q
\frac{1}{1-\sigma_r(w)^2}
= \frac{2}{\Delta(I-w^*w)}.
\end{equation}
We thus conclude that
$$D\le 2 (e^{\alpha_1}+e^{\beta_1})     |\alpha-\beta|\le 8q
(e^{\alpha_1}+e^{\beta_1})
  \|\lambda\|_1\,\widetilde t
\cdot\frac{\sigma_1(w)}{\Delta(I-w^*w)}.$$
By this this estimate and H\"olders inequality we obtain
\begin{align}\label{main-absch-1}
&\bigl| \phi_{\lambda-i\rho}^p(t)-\psi_{\lambda-i\rho^A}(t)\bigr|\le
\\
&\quad \le \, 8q\|\lambda\|_1\,\widetilde t \cdot
\int_{B_q\times U_0(q,\mathbb F)} (e^{\alpha_1}+e^{\beta_1})\frac{\sigma_1(w)}
 {\Delta(I-w^*w)}dm_p(w) du
\notag\\
&\quad \le 8q\|\lambda\|_1\,\widetilde t \cdot\Bigl(
\int_{B_q}\frac{\sigma_1(w)^{2n}}
 {\Delta(I-w^*w)^{2n}} dm_p(w)\Bigr)^{1/2n} \times
\notag\\
&\quad\quad\quad\times \Bigl[\Bigl(\int_{B_q\times U_0(q,\mathbb
F)}e^{\frac{2n}{2n-1}\alpha_1}
dm_p(w) du\Bigr)^{\frac{2n-1}{2n}}\, + \,
\Bigl( \int_{B_q \times U_0(q,\mathbb F)}
e^{\frac{2n}{2n-1}\beta_1}
dm_p(w) du\Bigr)^{\frac{2n-1}{2n}}\Bigr].
\notag\end{align}
In view of (\ref{wdh-int-rep1}) and (\ref{wdh-int-rep2}),
 the $[\ldots]$-term in the last two lines is equal to
$$\phi_{\frac{2n}{2n-1}i\text{Im}\lambda-i\rho}^p(t)^{\frac{2n-1}{2n}} +
\psi_{\frac{2n}{2n-1}i\text{Im}\lambda-i\rho^A}(t)^{\frac{2n-1}{2n}},$$
and the lemma follows.
\end{proof}

Eq.~(\ref{absch1}) in Theorem \ref{absch-q-p-infty} is now a 
consequence of Lemma \ref{lemma-eins3} and the
following result:

\begin{lemma}\label{lemma-eins4}
For each $n\in\mathbb N$ there is a  constant $C=C(\mathbb F, q,n)>0$ such that
for all $p\ge 2q$,
$$R(p):= \int_{B_q}\frac{\sigma_1(w)^{2n}}
 {\Delta(I-w^*w)^{2n}} dm_p(w) \,\leq \,\frac{C}{p^n}\,.$$
\end{lemma} 

\begin{proof} 
We  transform the integral in the lemma. The
diffeomorphism $P:B^q\to B_q$
introduced in Remark \ref{degen-case-2q-1}, where $B$ is the
ball $B:=\{y\in\mathbb F^q:\> \|y\|_2<1\}.$ We recall from \cite{R1}
that for $w=P(y_1, \ldots, y_q), $ one has
$\, \Delta(I-w^*w) = \prod_{j=1}^q (1 -\|y_j\|_2^2).\,$
With \eqref{image_meas} in mind, we obtain
\begin{equation}\label{main-absch-3}
 R(p)=\frac{1}{\kappa_{pd/2}}\cdot
\int_{B^q}  \sigma_1(P(y_1,\ldots,y_q))^{2n}\cdot
 \prod_{j=1}^{q}(1-\|y_j\|_2^2)^{ d(p-q-j+1)/2-1-2n}d(y_1,\ldots,
y_q).
\end{equation}
Moreover, the $j,j$-element $(ww^*)_{jj}$ of $ww^*$ satisfies
$$(ww^*)_{jj}= y_j(I-y_1^*y_1)^{1/2}\ldots
(I-y_{j-1}^*y_{j-1})^{1/2}(I-y_{j-1}^*y_{j-1})^{1/2}
\ldots(I-y_1^*y_1)^{1/2}y_j^*.$$
As the hermitian matrix $I-y^*y\,$ has eigenvalues in $[0,1]$,
 it follows readily that
$0\le (ww^*)_{jj}\le \|y_j\|_2^2$ and hence
\[\sigma_1(w)^2 \le \sum_{j=1}^q (ww^*)_{jj}\le \sum_{j=1}^q \|y_j\|_2^2 .\]
Therefore,
\[\sigma_1(w)^{2n} \le C\cdot \sum_{j=1}^q \|y_j\|_2^{2n}\]
with some constant $C>0$.
 This leads to the estimate
\begin{equation}\label{main-absch-4}
R(p)\le  \frac{C}{\kappa_{pd/2}} 
\sum_{j=1}^q 
\int_{B^q}   \|y_j\|_2^{2n}\cdot
 \prod_{r=1}^{q}(1-\|y_r\|_2^2)^{d(p-q-r+1)/2 \,-1-2n}d(y_1,\ldots, y_q).
\end{equation}
Using polar coordinates, we obtain for $y= y_r$ and arbitrary $\alpha>0$ 
that
\[\int_B (1-\|y\|_2^2)^{\alpha-1}\> dy = \,\omega_{dq} \int_0^1
x^{dq-1}(1-x^2)^{\alpha-1}\> dx
= \omega_{dq}\cdot \frac{\Gamma(\alpha)\,\Gamma\bigl(\frac{dq}{2}\bigr)}{2\cdot
\Gamma\bigl(\alpha + \frac{dq}{2}\bigr)}\]
and
\[\int_B\|y\|_2^{2n} (1-\|y\|_2^2)^{\alpha-1}\> dy =\,
\omega_{dq} \int_0^1 x^{dq-1+2n}(1-x^2)^{\alpha-1}\> dx
=
\omega_{dq}\cdot\frac{\Gamma(\alpha)\,\Gamma\bigl(n+\frac{dq}{2}\bigr)}{
2\cdot\Gamma\bigl(\alpha+n +\frac{dq}{2}\bigr)}\]
with the surface measure $\omega_{dq}:=vol(S^{dq-1})$ of the unit sphere in
$\mathbb
R^{dq}$ as 
 normalization constant. 
These formulas yield that
\begin{align}\label{kappa-explicit}
\kappa_{pd/2}=\,&\int_{B^q}
 \prod_{r=1}^{q}(1-\|y_r\|_2^2)^{ d(p-q-r+1)/2-1 }d(y_1,\ldots, y_q)\notag\\
=\,&\Bigl(\frac{\omega_{dq}}{2}\cdot
\Gamma\bigl(\frac{dq}{2}\bigr)\Bigr)^{q}\cdot
\prod_{r=1}^q\frac{\Gamma\bigl(\frac{d}{2}(p-q-r+1)\bigr)}{\Gamma\bigl(\frac{d}{
 2} (p-r+1) \bigr)}
\end{align}
and
\begin{align}
&I_j(p):= \frac{1}{\kappa_{pd/2}}\cdot\int_{B^q}  \|y_j\|_2^{2n}\cdot
 \prod_{r=1}^{q}(1-\|y_r\|_2^2)^{ d(p-q-r+1)/2 -1 -2n}\,d(y_1,\ldots, y_q)
=\notag\\
&= \frac{\Gamma\bigl(n +
\frac{dq}{2}\bigr)}{\Gamma\bigl(\frac{dq}{2}\bigr)}\cdot\frac{\prod_{r=1}^q
\Gamma\bigl(\frac{d}{2}(p-q-r+1)-2n)}{\Gamma\bigl(\frac{d}{2}(p-j+1)
-n\bigr)\cdot\prod_{r\not=j}\Gamma\bigl(\frac{d}{2}(p-r+1)-2n\bigr)} \cdot
\prod_{r=1}^q
\frac{\Gamma\bigl(\frac{d}{2}(p-r+1)\bigr)}{\Gamma\bigl(\frac{d}{2}
(p-q-r+1)\bigr) } \, .
\notag
\end{align}
>From the asymptotics of the gamma function we
obtain for $p\to \infty$ the asymptotic equality
\[ I_j(p) \,\sim \, \frac{\Gamma\bigl(n +
\frac{dq}{2}\bigr)}{\Gamma\bigl(\frac{dq}{2}\bigr)}\cdot
\bigl(\frac{dp}{2})^{-n}\quad (p\to \infty). \]
This implies that $R(p)$ is of order
$O(p^{-n})$ for $p\to\infty$. 
\end{proof}

The proof of Theorem \ref{absch-q-p-infty} is now complete. 

\section{Convergence to Bessel functions of type B}

In this section we consider the Heckman-Opdam function $\phi_{\lambda}^p$ for fixed $p\in \mathbb R$ with $p\ge 2q-1$
in a scaling limit. More precisely, we use the integral repesentation of Theorem \ref{int-rep} in order to derive 
convergence of the rescaled
 functions $\phi_{n\lambda-i\rho}^p(t/n)$ for $n\to \infty$  to
Dunkl-type Bessel functions associated with root system $B_q$. While such asymptotics are 
well-known in a general context from the asymptotics of the hypergeometric
system, we here obtain a precise estimate
 for the rate of convergence.

To explain the result,
 let us first recall some facts on Bessel functions from \cite{FK},\cite{Ka} and
\cite{R1}.

\begin{Bessel}\label{bessel-def}
Let ${\bf m}=(m_1,\ldots,m_q)$  be a partition of length $q$ with integers
$m_1\ge
m_2\ge\ldots\ge m_q\ge0$ and let $|{\bf m}|:=m_1+\ldots +m_q$.
For $x\in\mathbb C$ and a parameter $\alpha>0$, 
the generalized Pochhammer symbol is given by
\begin{equation}
 (x)_{\bf m}^\alpha \,=\,\prod_{j=1}^q
\bigl(x-\frac{1}{\alpha}(j-1)\bigr)_{m_j}.
\end{equation}
\noindent
For $\mathbb F= \mathbb R, \mathbb C, \mathbb H$  with $d=\text{dim}_\mathbb R \mathbb F $ 
and  partitions ${\bf m}$,
 the  spherical polynomials $\Phi_{\bf m}$ are defined by 
$$ \Phi_{\bf m} (x) = \int_{U_q} \Delta_{\bf m}(uxu^{-1})du
\quad\text{for }\,  x\in M_{q}(\mathbb F)$$
where  $\Delta_{\bf m}$ is the power function of Eq.~(\ref{power-function}).
We also consider the renormalized polynomials
$Z_{\bf m} = c_{\bf m}\cdot \Phi_{\bf m}$
with certain normalization constants $c_{\bf m}>0$ which are 
 characterized by the formula
\begin{equation}\label{traceid}
(\text{tr} \,x)^k \,=\, \sum_{|{\bf m}|=k} Z_{\bf m}(x)
\quad\quad{\rm for}\>\> k\in \mathbb N_0, \, x\in M_{q}(\mathbb F).
\end{equation}

 By construction, the $ \Phi_{\bf m}$ and  $Z_{\bf m}$
 are invariant under conjugation by $U(q,\mathbb F)$ and thus depend only on
 the eigenvalues of their argument.
More precisely, for a Hermitian matrix $x\in M_{q}(\mathbb F)$ with
eigenvalues 
$\xi = (\xi_1, \ldots, \xi_q)\in \mathbb R^q$, we have $Z_{\bf m}(x) = 
C_{\bf m}^\alpha(\xi)$
where the $C_{\bf m}^\alpha$ are the Jack polynomials of index $\alpha:=2/d$; see
Section XI of  \cite {FK} and references cited there.
The Jack polynomials are homogeneous
 of degree $|{\bf m}|$ and symmetric in their arguments.

Following Kaneko \cite{Ka} (see also Section 2.2 of \cite{R1}) we
define Bessel functions in two arguments
\begin{equation}\label{def-bessel}
J_\mu(\xi,\eta):= \sum_{\bf m} \frac{(-1)^{|{\bf m}|}}{(\mu)_{\bf m}^\alpha
 |{\bf m}|!} \cdot
\frac{C_{\bf m}^\alpha(\xi)C_{\bf m}^\alpha(\eta)}{C_{\bf m}^\alpha(1,\ldots,1)}
\end{equation}
for $\mu\in\mathbb C$ with $(\mu)_{\bf m}^\alpha\ne0$ for all partitions 
${\bf  m}$ and
with fixed parameter $\alpha:=2/d$.
A comparison of (\ref{def-bessel}) with the explicit form of the Dunkl-type Bessel
 functions $J_k^B$ associated with root system $B_q$ in \cite{BF} 
shows that 
the Bessel function $J_\mu$ can be expressed in terms of $J_k^B$ as
$$J_\mu\bigl(\frac{\xi^2}{2},\frac{\eta^2}{2}\bigr)=J_k^B(\xi, i\eta),$$
with the  multiplicity parameter $k:=k(\mu,d):=(\mu-(q-1)d/2 -1/2, d/2)$.
 For the details see Section 4.3 of \cite{R1} and \cite{O} for the general
context.
\end{Bessel}

For certain indices $\mu$, the Bessel functions $J_\mu$ appear as the spherical
functions of the Euclidean-type symmetric spaces $G_0/K$ where $K=U(p,\mathbb F)\times U(q,\mathbb F)$
and $G_0= K\ltimes M_{p,q}(\mathbb F)$ is the Cartan  motion group associated with the
Grassmannian  $\mathcal G_{p,q}(\mathbb F)$. 
The double coset space $G_0//K$ is naturally identified with the Weyl chamber 
$C_q$, with $t\in C_q$ corresponding to the double coset of $(I_p, I_q, \underline t\,) \in G_0$. So we may consider biinvariant functions on 
 $G_0$ as functions on $C_q$. It is  well known  (see Section 4 of \cite{R1})
that the spherical functions of $(G_0,K)$ are given in terms of the
 Bessel function $J_\mu$ as follows:

\begin{proposition}\label{class-spher-2} The spherical functions of $(G_0,K)$
are  given by
the Dunkl-type Bessel functions
 \[\widetilde\phi_\lambda^{\,p}(t):=J_k^B(t, i\lambda)=
J_\mu\bigl(\frac{\lambda^2}{2}, \frac{t^2}{2}\bigr), \quad \lambda\in \mathbb C^q \]
with  $\mu:=pd/2$ and $k$ as in Section 
\ref{bessel-def}. 
Moreover, $\widetilde\phi_\lambda^{\,p}$ is bounded precisely for
 $\lambda\in \mathbb R^q$. 
\end{proposition}

The spherical functions of  $(G_0,K)$ with  dimension 
parameters
$p\ge 2q$  admit a Harish-Chandra integral
 representation which can be extended by Carlson's theorem to all
 real parameters $p> 2q-1$ and thus to the corresponding indices $\mu$. This leads to the following

\begin{proposition}\label{Bessel-inte}
For all real parameters $p>2q-1$
 and all  $\,t\in C_q$ and $\lambda\in
\mathbb C^q$,
\begin{equation}\label{int-rep-Bessel}
\widetilde\phi_\lambda^{\,p}(t)=\int_{B_q}\int_{U_0(q,\mathbb F)} e^{-i\,\text{Re}\,\text{tr}(w
  \underline t u \underline\lambda\,)}dm_p(w)du
\end{equation}
with the probability measure $m_p\in M^1(B_q)$ of Eq.~(\ref{measure-mp}).
Moreover, for  $p=2q-1$ and with the notations of Remark
 \ref{degen-case-2q-1},
\begin{equation}\label{int-rep-Bessel-sing}
\widetilde\phi_\lambda^{\,p}(t)=
\frac{1}{\kappa_{(2q-1)d/2}}\int_{B^{q-1}\times S}\int_{U_0(q,\mathbb F)}
e^{-i\,\text{Re}\,\text{tr}( P(y)
  \underline t u\underline \lambda\,)}
\cdot\prod_{j=1}^{q-1}(1-\|y_j\|_2^2)^{q-1-j}dy_1\ldots
dy_{q-1}\> d\sigma(y_q)\> du.
\end{equation}
 \end{proposition}

\begin{proof} 
For $p> 2q-1$ and $\lambda\in C_q$, the first formula
is immediate by a combination of the integral representations 
(3.12) and (4.4) in \cite{R1} (in the latter, integration over $U(q, \mathbb F)$ my be replaced 
by integration over $U_0(q, \mathbb F)$.)
 The general case $\lambda\in \mathbb C^q$ then
follows by analytic continuation. 

 The singular limit case $p= 2q-1$ can be
derived in the same way as in  \cite{R1}; see also Remark
\ref{degen-case-2q-1}.
We omit the details.
\end{proof}

A comparison of these integral representations for 
the Bessel functions $\widetilde\phi_\lambda^{\,p}$ with the integral
representation 
for the Heckman-Opdam functions $\phi_\lambda^p$ of Section 2 leads to the
following theorem, which is the main result of this section.

\begin{theorem}\label{vgl-Bessel}
For each compact subset $K\subset \mathbb R^q$ 
 there exists a constant $C=C(K)>0$ such that for all
 $p\in\mathbb R$ with $p\geq 2q-1$,
all $\lambda\in \mathbb R^q,  \, t\in K$, and  all $n\in \mathbb N$, 
$$|\phi_{n\lambda-i\rho}^p(t/n)-\widetilde\phi_\lambda^{\,p}(t)|\leq\, C\cdot\frac{ \|\lambda\|_1}{n}\,.$$
Here again, $\|\lambda\|_1=|\lambda_1|+\ldots+|\lambda_q|$.
\end{theorem}

\begin{proof} We only give a proof for the non-degenerate case $p> 2q-1$.
The case $p= 2q-1$ follows in the same way from
(\ref{int-rep-Bessel-sing}) and Remark  \ref{degen-case-2q-1}.

We substitute  $w\mapsto -u^*w^*$ in the integral
  (\ref{int-rep-Bessel})
and obtain
$$\widetilde\phi_\lambda^{\,p}(t)=\,\int_{B_q}\int_{U_0(q, \mathbb F)} e^{i\cdot Re\, tr(u^*
  w^*\underline t\, u\underline \lambda\,)}dm_p(w)du\,.$$
Moreover, denoting the trace of the upper left $r\times r$-block of a
$q\times q$-matrix by $tr_r\,,$ we have
\begin{align*}
Re\, tr(u^*  w^*\underline t u\underline\lambda\,)
&= \,\frac{1}{2}\cdot\sum_{r=1}^q (u^*((\underline tw)^*+\underline tw)u)_{rr}\cdot \lambda_r\\
&=\,\sum_{r=1}^q \bigl[ \, tr_r(u^*((\underline tw)^* + \underline tw)u) -
 tr_{r-1}(u^*((\underline tw)^* + \underline tw)u)\bigr] \cdot \lambda_r/2\\
&=\,\sum_{r=1}^q  tr_r(u^*((\underline tw)^*+ \underline tw)u) \cdot(\lambda_r-\lambda_{r+1})/2
\end{align*}
with $\lambda_{q+1}:=0$. Thus,
$$\widetilde\phi_\lambda^{\,p}(t)=\,\int_{U_0(q,\mathbb F)\times B_q}\, \prod_{r=1}^q exp\left( i\cdot
 tr_r(u^*((tw)^*+tw)u) \cdot( \lambda_r-\lambda_{r+1})/2\right) dm_p(w)du.
$$
Further, according to  Theorem \ref{int-rep}, 
$$\phi_{n\lambda-i\rho}^p(t/n)=\int_{U_0(q,\mathbb F)\times B_q}\, \prod_{r=1}^q
 \Delta_r(g_{t/n}(u,w))^{in( \lambda_r-\lambda_{r+1})/2}\,dm_p(w)du$$
with the positive definite matrix
$$g_{t/n}(u,w) = u^*(\cosh(t/n)+ 
\sinh(t/n)\cdot w)^*(\cosh(t/n)+ \sinh(t/n)\cdot w)
u.$$
Using the well-known estimate 
$$\left|\prod_{r=1}^q a_r-\prod_{r=1}^q b_r\right|\le\sum_{r=1}^q|a_r-b_r| 
\quad\quad\text{for }\, a_r,b_r\in \{z\in\mathbb C:\> |z|=1\},$$
we obtain
\begin{align*}
C:=\,&\,|\phi_{n\lambda-i\rho}^p(t/n)-\widetilde\phi_\lambda^{\,p}(t)|\\
\leq\,&\,\sum_{r=1}^q \int_{U_0(q,\mathbb F)\times B_q} \Bigl| \Delta_r(g_{t/n}(u,w))^{in(
  \lambda_r-\lambda_{r+1})/2}\\
 &\qquad\qquad\qquad\qquad -
exp\left( i\cdot
 tr_r(u^*((\underline tw)^*+ \underline tw)u) \cdot( \lambda_r-\lambda_{r+1})/2\right)\Bigr|\,
dm_p(w)du.
\end{align*}
Further, by the inequality
$$|e^{ix}-e^{iy}|\le\sqrt 2 \cdot |x-y|
\quad\quad\text{for }\, x,y\in\mathbb R,$$
we obtain
$$C\leq\, \frac{1}{\sqrt 2}\sum_{r=1}^q| \lambda_r-\lambda_{r+1}|\cdot C_r$$
with
$$C_r:= \int_ {U_0(q,\mathbb F)\times B_q}\bigl| n \ln \Delta_r(g_{t/n}(u,w)) -
tr_r(u^*((\underline tw)^*+ \underline tw)u)\bigr|\, dm_p(w)du.$$
We now write
$g_{t/n}(u,w)=I+ A/n + H/n^2\,$
with $A:=u^*((tw)^*+tw)u\,$ and some Hermitian matrix $H=H(u,w,t,n)$ which
stays in a compact subset of $M_q$ for $(u,w,t,n)\in U_0(q, \mathbb F)\times B_q\times K\times \mathbb N.$
Therefore,
$$
 n \ln \Delta_r(g_{t/n}(u,w))=  n \ln \Delta_r(I+ A/n
 +H/n^2)=n\ln \bigl(1+tr_r(A)/n +h/n^2\bigr)$$
with some constant $\,h=h(u,w,t,n)\in \mathbb C\,$ which remains bounded for the arguments under consideration.
 Using the power series for
$\ln(1+z)$, we get
$$ n \ln \Delta_r(g_{t/n}(u,w)) - tr_r(A) =O(1/n)
\quad\quad\text{for }\,
n\to\infty,$$
uniformly in $u, w$ and $t\in K$. This yields the assertion.
\end{proof}

\begin{remarks}
\begin{enumerate}
\item[\rm{(1)}] Similar to the results in Section 4, Theorem \ref{vgl-Bessel}
  can be extended from $\lambda\in\mathbb R^q$ to $\lambda\in\mathbb C^q$ with
  suitable exponential bounds on the right side of the estimate.
\item[\rm{(2)}] We point out that one may also compare the integral representation for the
spherical
  functions of the symmetric spaces 
$GL(q, \mathbb F)/U(q, \mathbb F)$ in Section 3 with the 
integral representation for the spherical functions $\widetilde\psi_\lambda$
  of $(U(q,\mathbb F)\ltimes H_q(\mathbb F), U(q,\mathbb F))$, where  
 $U(q,\mathbb F)$ acts by conjugation on the space $H_q(\mathbb F)$ of
  all Hermitian $q\times q$-matrices. In this case, the methods of the preceding proof 
  lead to a
  result analogous to that of Theorem \ref{vgl-Bessel}. Moreover, for real spectral variables $\lambda$ it is possible to
 combine
  this result with  Theorems \ref{vgl-Bessel} and \ref{main-p-infty}(2), in order to
  obtain a convergence result for the Dunkl-type Bessel functions
  $\widetilde\phi_\lambda^{\,p}$ to  the functions $\widetilde\psi_\lambda$ for $p\to\infty$ with explicit
  error bounds, similar to Theorem \ref{main-p-infty}(2).
However, these results will be weaker than those which 
 were derived directly in  \cite{RV2}.
\end{enumerate}
\end{remarks}

\section{Appendix: On convex hulls of Weyl group orbits}

In this appendix we present a proof of Lemma \ref{conv-hull}.
We start with some general facts, where we assume that $R$
is a crystallographic root system of rank $q$ in a Euclidean vector space
$(V, \langle \,.\,\rangle)$ with Weyl group $W$. We fix a closed Weyl chamber
$C_q$ for $R$ and denote by $\alpha_1, \ldots, \alpha_q\subset
R$ the simple roots associated with $C_q\,.$ We further introduce the dual cone
\[ C_q^+ := \{ x\in V: \langle x, y\rangle \geq 0\}.\]
It is well-known (see e.g. Lemma IV.8.3. of \cite{Hel}) that for
each $x\in C_q^+,$
\begin{equation}\label{cone_section}
 co(W.x) \cap C_q = C_q \cap (x-C_q^+).
\end{equation}

\begin{lemma}\label{ball_inside} Suppose that $R$ is irreducible.\parskip=-1pt
\begin{enumerate}\itemsep=-1pt
 \item[\rm{(1)}] Let  $x,y\in C_q\setminus \{0\}.$ Then $\,\langle x,
y\rangle >0.$
\item[\rm{(2)}] There exists a constant $\epsilon_0 >0$ such that
 the ball $B_{\epsilon_0}(0) = \{ x\in V: \|x\|<\epsilon_0\}$ is contained in
$co(W.x)$ for each $x\in C_q$ with $\|x\|=1.$
\end{enumerate}
\end{lemma}

\begin{proof} (1)
 Let $\lambda_1, \ldots, \lambda_q\in V$ denote the fundamental weights
associated with $\alpha_1, \ldots, \alpha_q$, defined by $\langle\lambda_j,
\alpha_i^\vee \rangle = \delta_{ij}$ with $\,\alpha_i^\vee = 2\alpha_i/\langle
\alpha_i, \alpha_i \rangle.$ Then both $x$ and $y$  can be written as linear
combinations of the $\lambda_i$ with non-negative coefficients (see \cite{Hu1},
Section 13.1). By our assumption on $R$ and Section 13 of \cite{Hu1}, 
the weights $\lambda_i$ satisfy
$\langle \lambda_i, \lambda_j \rangle >0$ for all $i,j$. We therefore obtain
that $\langle x, y\rangle >0.$

(2) Let $C_q^1 := \{ x\in C_q:\|x\|=1\}$ and consider the continuous mapping
$\,(x,y)\mapsto \langle x,y \rangle$  on the compact set
$C_q^1\times C_q^1$. By part (1), there exists some $\epsilon_0 >0$ such that
\[\langle x,y\rangle >\epsilon_0 \quad \text{for all }\,x, y\in C_q^1.\]
Now fix $x\in C_q^1\,.$
We claim that $\, B_{\epsilon_0}(0) \subseteq co(W.x).$
For this, let $\,z\in B_{\epsilon_0}(0) \cap C_q\,.$ Then for each $y\in C_q^1,$
we have
\[ \langle z,y\rangle < \epsilon < \langle x,y \rangle .\]
This shows that $\, x-z\in C_q^+$ and  $z\in x-C_q^+.$ In view of
\eqref{cone_section}, we thus obtain
\[ B_{\epsilon_0}(0) \cap C_q \subseteq \,co(W.x)\cap C_q.\]
The claim is now immediate.
\end{proof}

We now fix some $\rho\in C_q $ and  consider the compact convex set
\[ K := co(W.\rho)\cap C_q\,.\]
We collect some simple facts on the extreme points of
$K$.

\begin{lemma}\label{basic-convexity} 
\begin{enumerate}\itemsep=-1pt

\item[\rm{(1)}] The topological boundary $\partial C_q$ of $C_q$ is
  contained in the union of the reflecting hyperplanes
  $H_{\alpha_1},\ldots,H_{\alpha_q}  $ associated with  the simple reflections
 $\sigma_{\alpha_1},\ldots, \sigma_{\alpha_2}$, and  $C_q$
 is the intersection of  $q$ closed half-spaces.
\item[\rm{(2)}]  The closed cone $\rho-C^+_q$ is also the intersection of
  $q$ closed  half-spaces corresponding to  hyperplanes 
$H^+_{1},\ldots,H^+_{q}$.
\item[\rm{(3)}] $K$ is a compact convex polytope which is obtained as the
intersection
  of $2q$ closed  half-spaces. Moreover, if $x$ is an extreme point of
  $K\,,$ then $x=0$, $x=\rho$, or $x\in  \partial C_q\cap
  \partial(co(W.\rho))$.
 \item[\rm{(4)}] If $x\in K$ is an extreme point different from $0$ and
$\rho$, then
   there exists $\,k\in\{1,\ldots,q-1\}$ such that $x$ is contained in the
   $q$-fold intersection of $k$   hyperplanes $H_{\alpha_j}$ and $q-k$
 hyperplanes $H_l^+$.
\end{enumerate}
\end{lemma}

\begin{proof} \begin{enumerate}\itemsep=-1pt
\item[\rm{(1)}] See Section 10.1 of \cite{Hu1}.
\item[\rm{(2)}] This follows from (1) and the definition of the dual cone.
\item[\rm{(3)}] The first statement is clear by (1), (2) and
  (\ref{cone_section}).
For the second statement, consider some extreme point  $x$ of
 $K=C_q\cap (\rho-C^+_q)$. If
  $x$ is contained in the interior of $C_q$, then it is easily checked that
   $x$ has to be an extreme point of the cone $\rho-C^+_q\,$ which implies
  $x=\rho$. Moreover, if $x$ is contained in the interior of $\rho-C^+_q$ then
  by the same reasons,  $x$ has to be an extreme point of $C_q$ and hence
  $x=0$.This yields the assertion.
\item[\rm{(4)}] This follows from (3).
\end{enumerate}\end{proof}

\begin{lemma}\label{combine}
 Let $W_1,W_2$ be reflection groups acting on $V_1$ and $V_2$
respectively. Let $\rho_i\in V_i$ and
$a_i\in co(W_i.\rho_i)$ for $ i= 1,2.$ Then
$(a_1,a_2)\in V_1\times V_2\,$ satisfies
$(a_1,a_2)\in co((W_1\times W_2)(\rho_1,\rho_2))$.
 \end{lemma}

\begin{proof}  For $i=1,2$, we have $\,a_i=\sum_{w_i\in W_i} \lambda_{w_i}^i
  w_i\rho_i\,$ with $\,\lambda_{w_i}^i\ge0$ and $\sum_{w_i\in W_i}
\lambda_{w_i}^i =1$.
Therefore,
$$(a_1,a_2)=\sum_{w_1\in W_1}\sum_{w_2\in
  W_2}\lambda_{w_1}^1\lambda_{w_2}^2 \cdot(w_1\rho_1, w_2\rho_2)$$
as claimed.
\end{proof}

We finally turn to the proof of Lemma \ref{conv-hull}.
As for Weyl groups of type $B$ the mapping $x\mapsto -x$ on $\mathbb R^q$
corresponds
to the action of some Weyl group element, Lemma \ref{conv-hull} is a
 consequence of part (1) of the following  result.

\begin{proposition}\label{contained}
Consider a root system $R$ of rank $q$ in a Euclidean space $V$  with
reflection group $W\subset O(V)$ and
a fixed closed Weyl chamber $C_q$ in one of the  following cases:
\begin{enumerate}\itemsep=-1pt
\item[\rm{(1)}]  $R=B_{q}$ and $V=\mathbb R^q$, or
\item[\rm{(2)}] 
$R=A_{q}$ and the symmetric group  $W=S_{q+1}$ acts either on 
  $V=\mathbb R^{q+1}$ or  $V=(1, \ldots, 1)^\perp\subset \mathbb R^{q+1}\,$ in
  a non-effective or effective  way.
\end{enumerate}
Then there exists some $\epsilon_0>0$ (depending on $R$) 
such that for all $\,0\le \epsilon\le \epsilon_0$, $\rho\in C_q$, and  $y\in
co(W.\rho)\cap C_q$,
$$(1+\epsilon)y-\epsilon\rho\in co(W.\rho).$$
\end{proposition}

Notice that for fixed $y$, the point $\,(1+\epsilon)y-\epsilon\rho = y+\epsilon(y-\rho)$
is opposite to $\rho$ with respect to $y$ on the line through $y$ and $\rho$,
with distance $\,\epsilon \|y-\rho\|$ from $y$. 
In case $\epsilon=1$, it is obtained from $\rho$ by 
reflection in $y$. 

\smallskip

For the root systems $A_1, B_1$ and $B_2$ the
maximal parameter  is $\epsilon_0 = 1$ while in  the reduced $A_2$-case
the maximal parameter is  $\epsilon_0= 1/2.$  In fact, the cases  $A_1, B_1$ are trivial, while the cases  $A_2$,
$B_2$
follow easily from the following diagrams:

\noindent
\begin{minipage}[c]{1.00\textwidth}
\centering
	\includegraphics[width=0.80\textwidth]{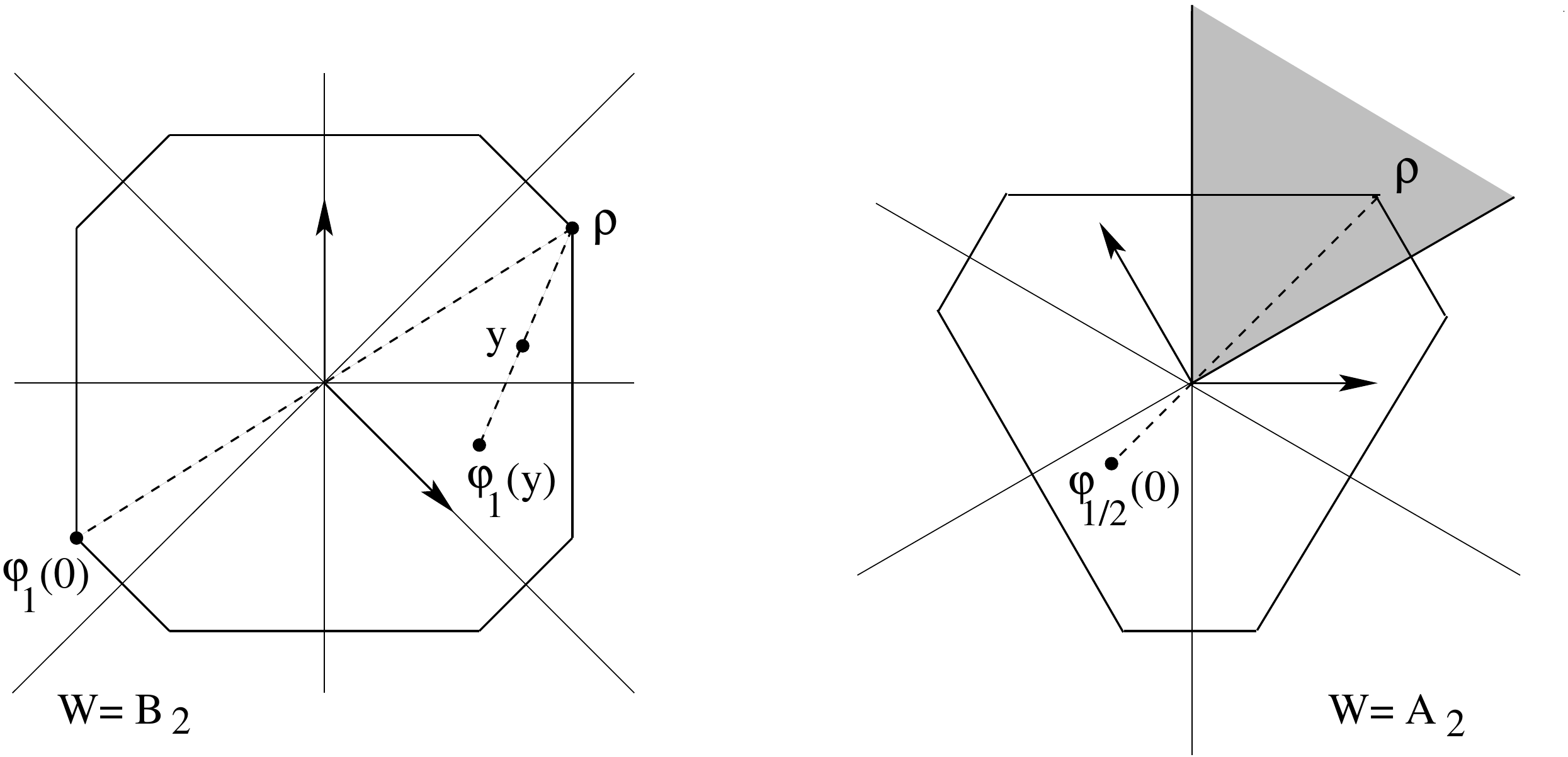}

 \end{minipage}

\vspace{13pt}

\begin{proof}[Proof of Proposition \ref{contained}]
For the proof of the general case, we fix $\rho\in C_q$ 
and
consider 
\[K:=co(W.\rho)\cap C_q\]
 as well as for $\epsilon>0$, its image
$K_\epsilon:=\phi_\epsilon(K)\,$ under the affine mapping
\[\phi_\epsilon:y\mapsto
(1+\epsilon)y-\epsilon \rho\,.\]
Clearly, $K_\epsilon$ is again compact and convex, and  $\phi_\epsilon$
maps
extreme points of $K$ onto extreme point of  $K_\epsilon$. For the
proof of
Proposition \ref{contained} it suffices to prove that extreme
points of $K$ are mapped to points in $co(W.\rho)$ for
 $\epsilon\in[0,\epsilon_0]$ with $\epsilon_0>0 $ sufficiently small.
 For the proof of this statement, we  may assume 
that in addition $\|\rho\|_2=1$ holds, and that, by a continuity argument, 
 $\rho$ is contained in the interior of $C_q$.

We  prove Proposition \ref{contained} by induction on $q$
first for the $A_q$-cases  and then for  $B_q$,
 where  the $A$-cases are used. The  proposition is clear for
 $A_1$ and $B_1$.
Let  $y\in K$ be an  extreme point.
By Lemma \ref{basic-convexity}(4), we have 3 cases of extreme points:
\smallskip

If $y=\rho$, then $\phi_{\epsilon}(\rho)=\rho$, and the claimed statement
 is trivial.

Moreover, if $y=0$, then $\phi_{\epsilon}(0)=-\epsilon\rho$, and the statement
follows in all cases with $\epsilon_0>0$ as in Lemma
\ref{ball_inside}(2).

\medskip
We now turn to the third case. Assume first that  
$S_{q+1}$ acts on the vector space $V_q:=(1,\ldots,1)^\perp\subset
\mathbb R^{q+1}$ where
 $C_q$ is the closed Weyl chamber associated with the simple roots
$$\alpha_1:=e_1-e_2,\alpha_2:=e_2-e_3,\ldots,  \alpha_q:=e_{q}-e_{q+1},$$
and $e_1,\ldots,e_{q+1}$ is the standard basis of  $\mathbb R^{q+1}$.
We first study the extreme point $x_0\in C_q\cap  co(W.\rho)$
 contained
 in the intersection of the hyperplanes
$H_{\alpha_1},\ldots, H_{\alpha_{q-1}}\subset V_q$ and the  hyperplane 
$$H:=\{x\in V_q:\> \langle x,e_{q+1}\rangle=\langle
\rho,e_{q+1}\rangle\}$$
which contains the $q$ affinely independent points 
$\rho, \sigma_{\alpha_1}(\rho),\ldots,  \sigma_{\alpha_{q-1}}(\rho)$ (notice
that $\rho$ is in the interior of $C_q$). 
We observe that $S_q$ as a subgroup of $S_{q+1}$ acts on  $H$ by
permutations of the first $q$ components.
We now identify $H$ with the vector space 
$V_{q-1}\subset \mathbb R^q$ via the affine mapping
$$(x_1,\ldots, x_q,\rho_{q+1})\mapsto 
(x_1-\rho_{q+1}/q,\ldots,x_q- \rho_{q+1}/q ).$$
In terms of this identification, the action of  $S_q$ on $H$ is just
 the usual action of 
 $S_q$ on $V_{q-1}$ with the simple reflections 
$ \sigma_{\alpha_1},\ldots,  \sigma_{\alpha_{q-1}}$.  
We now  regard the points
$\rho, x_0, \phi_\epsilon(x_0), \sigma_{\alpha_1}(\rho),$ $\ldots,
 \sigma_{\alpha_{q-1}}(\rho)\in H$
 as points of  $V_{q-1}$ and may 
apply the assumption in the induction for $A_{q-1}$.  This shows that
 $\phi_{\epsilon_0}(x_0)$ is contained in $co(S_q.\rho)\subset
co(S_{q+1}.\rho)$
 for $\epsilon_0>0$
sufficiently small. 
This proves the claim for this extreme point $x_0$.

The case of the  extreme point  in the intersection of 
$H_{\alpha_2},\ldots, H_{\alpha_{q}}$ and the corresponding hyperplane 
 $H$ containing the $q$ points 
$\rho, \sigma_{\alpha_2}(\rho),\ldots,  \sigma_{\alpha_{q}}(\rho)$ can be
handled in the same way.

For the next  type of an extreme point, we fix $k=2,\ldots, q-1$ and define
$$S:=\rho_1+\ldots+\rho_k=-(\rho_{k+1}+\ldots+\rho_{q+1}).$$
We now consider the 
   extreme point
$x_0$ which is contained  in the intersection of the hyperplanes
$H_{\alpha_1},\ldots,H_{\alpha_{k-1}},H_{\alpha_{k+1}},\ldots,  H_{\alpha_{q}}$
and the
  hyperplane 
$$ H:=\{(x_1,\ldots,x_{q+1})\in \mathbb R^{q+1}: \> x_1+\ldots+x_k=S, 
\> x_{k+1}+\ldots+x_{q+1}=-S\}\subset V_{q}.$$
$H$ contains the affinely independent $q$ points 
$\rho, \sigma_{\alpha_1}(\rho),\ldots,  \sigma_{\alpha_{k-1}}(\rho), 
\sigma_{\alpha_{k+1}}(\rho), \ldots, \sigma_{\alpha_{q}}(\rho)$.
We write $H$ as $H:= H_1\times H_2$ with
$ H_1:=\{(x_1,\ldots,x_{k})\in \mathbb R^{k}: \> x_1+\ldots+x_k=S\}$ and
$ H_2:=\{(x_{k+1},\ldots,x_{q+1})\in \mathbb R^{q+1-k}: \> 
x_{k+1}+\ldots+x_{q+1}=-S\}$
 where the group $S_k\times S_{q+1-k}$ as a subgroup of $ S_{q+1}$ acts on $H$.
 We now identify $H_1$ with $V_{k-1}\subset\mathbb R^k$ 
 via the affine mapping
$$p_1:(x_1,\ldots, x_k)\mapsto 
(x_1-S/k,\ldots,x_k- S/k ),$$
and $H_2$ with $V_{q-k}\subset\mathbb R^{q+1-k}$ via
$$p_2:(x_{k+1},\ldots,x_{q+1})\mapsto 
(x_{k+1}+S/(q+1-k),\ldots,x_{q+1}+ S/(q+1-k) ).$$
In terms of this identification of $H$ with $V_{k-1}\times V_{q-k}$,
 the action of  $S_k\times S_{q+1-k}$ above  on $H$ is just
 the usual action of   $S_k\times S_{q+1-k}$ on  $V_{k-1}\times V_{q-k}$.
 We now consider the  Weyl chamber $C_{k-1}\subset V_{k-1} $
associated with the reflections $\sigma_{\alpha_1},\ldots,\sigma_{\alpha_{k-1}}$.
We see that $p_1(\rho)\in C_{k-1}$,
and that the points $$p_1(\rho), p_1(x_0), p_1(\phi_\epsilon(x_0)),
 \sigma_{\alpha_1}(p_1(\rho)),\ldots,
 \sigma_{\alpha_{k-1}}(p_1(\rho))\in V_{k-1} $$ are related in a way such
 that we may apply  the  induction assumption  for $A_{k-1}$. We conclude that
  $p_1(\phi_{\epsilon}(x_0))$ is contained in $co(S_{k}.p_1(\rho))$ for
 sufficiently small $\epsilon>0$.
In the same way, 
 $p_2(\phi_{\epsilon_0}(x_0))\in co(S_{q+1-k}.p_2(\rho))$ for
 sufficiently small $\epsilon>0$.
In view of Lemma \ref{combine} we conclude that there exists some
$\epsilon_0>0$ such that $\phi_{\epsilon}(x_0)\in
 co\bigl((S_k\times S_{q+1-k}).\rho\bigr)\subset
 co(S_{q+1}.\rho)$ for $0\leq \epsilon \leq\epsilon_0$ as
claimed.

\smallskip
We next study the  extreme points $x_0$ with the property that for some
 $k\in\{1,\ldots,q-1\} $, the point $x_0$ is contained in 
the $k$ reflecting hyperplanes
$H_{\alpha_{j_1}},\ldots,H_{\alpha_{j_k}}$ with $1\le j_1<\ldots <j_k\le q+1$
as well as in the $k$-dimensional affine subspace $H\subset V_q$ 
which
is spanned by the $k+1$ affinely independent points
 $\rho,\sigma_{\alpha_{j_1}}(\rho),\ldots,\sigma_{\alpha_{j_k}}(\rho)$. As in
the preceding case, we split the problem into several 
lower dimensional problems which can be handled 
separately by induction.
Again, by Lemma \ref{combine} we obtain some $\epsilon_0>0$ such that 
 $\phi_{\epsilon_0}(x_0)\in co(S_{q+1}.\rho)$ for $\epsilon \leq \epsilon_0.$
This completes the proof for the effective $A_q$-case, where $S_{q+1}$ acts on
$V_q$. This result implies also immediately the non-effective $A_q$-case,
 where $S_{q+1}$ acts on
$\mathbb R^{q+1}$.

\smallskip
We finally consider the case $B_q$ for $q>1$. We assume that $C_q$ is the
Weyl chamber associated with the simple roots
$$\alpha_1:=e_1-e_2,\alpha_2:=e_2-e_3,\ldots,  \alpha_{q-1}:=e_{q-1}-e_{q},
\alpha_q=e_q.$$
We here immediately study the general case where 
 for some
 $k\in\{1,\ldots,q-1\} $, the extreme point $x_0$ is contained in 
the $k$ reflecting hyperplanes
$H_{\alpha_{j_1}},\ldots,H_{\alpha_{j_k}}$ with $1\le j_1<\ldots <j_k\le q+1$
as well as in the affine subspace $H\subset \mathbb R^{q+1}$ of dimension $k$
which
is spanned by the $k+1$ points
 $\rho,\sigma_{\alpha_{j_1}}(\rho),\ldots,\sigma_{\alpha_{j_k}}(\rho)$.
As in
the  preceding case, we split the problem into several 
lower dimensional problems which can be handled either as a lower-dimensional
$B$-case  or as a known $A$-case. The proof is
again completed by induction and by use of
 Lemma \ref{combine}.
\end{proof}


\begin{thebibliography}{9999}\itemsep=-3pt
\bibitem[BF]{BF}  T.H. Baker, P.J. Forrester, The Calogero-Sutherland model and
generalized classical
polynomials. \textit{Comm. Math. Phys.} 188 (1997), 175--216.
\bibitem[D]{D} J.F. van Diejen, Asymptotics of multivariate orthogonal
 polynomials with hyperoctahedral symmetry. In: V.G. Kutznesov et al. (ed.):
 Jack, Hall-Littlewood and Macdonald poynomials. 
 American Mathematical Society. Contemp. Math. 417, 157-169 (2006).
\bibitem[FK]{FK} J. Faraut, A. Kor\'anyi, Analysis on symmetric cones. Oxford
 Science Publications,
 Clarendon press, Oxford 1994.
\bibitem[GV]{GV} R. Gangolli, V.S. Varadarajan, Harmonic analysis of spherical
functions on real reductive groups.
Springer-Verlag, Berlin Heidelberg 1988.
\bibitem[H]{H} G. Heckman, 
Dunkl Operators. S\'{e}minaire Bourbaki 828, 1996--97; Ast\'{e}risque 245
(1997), 223--246. 
\bibitem[HS] {HS} G. Heckman, H. Schlichtkrull, Harmonic Analysis and Special
Functions on Symmetric Spaces; 
Perspectives in Mathematics, vol. 16, Academic Press, California, 1994. 
\bibitem[Hel]{Hel} S. Helgason, Groups and Geometric Analysis. Mathematical
Surveys and Monographs, vol. 83, AMS 2000.
20 (1982), 69--85.
\bibitem[HJ] {HJ} R.A. Horn, C.R. Johnson, Topics in Matrix Analysis. 
Cambridge University Press 1991.
\bibitem[Hu] {Hu1} J.E. Humphreys, Introduction to Lie Algebras and
Representation Theory. Springer Verlag, 1973.
\bibitem[dJ] {dJ} M.F.E. de Jeu, Paley-Wiener theorems for the Dunkl transform.
Trans. Amer. Math. Soc. 358 (2006), 4225--4250.
\bibitem[J]{J} R.I. Jewett, Spaces with an abstract convolution of measures,
Adv. Math.  18 (1975), 1--101.
\bibitem[Ka]{Ka} J.~Kaneko, Selberg integrals and hypergeometric functions
  associated with Jack polynomials. \textit{SIAM J.~Math.~Anal.} 24
  (1993), 537--567.
\bibitem[K1]{K1} T. Koornwinder, Jacobi functions and analysis on noncompact
semisimple Lie groups.
In: \textit{Special Functions: Group Theoretical Aspects and Applications},
 Eds. Richard Askey et al.,  D.~Reidel,
 Dordrecht-Boston-Lancaster, 1984.
\bibitem[K2]{K2} T. Koornwinder, Jacobi polynomials of type BC, Jack
polynomials,
 limit transitions and $O(\infty)$. American Mathematical Society.
 Contemp. Math. 190, 283-286 (1995).
\bibitem[NPP]{NPP} E. K. Narayanan, A. Pasquale, S. Pusti: 
Asymptotics of Harish-Chandra expansions,
 bounded hypergeometric functions associated with root systems, and
 applications.
 Adv. Math. 252 (2014),  227--259.
\bibitem[O1]{O} E.M. Opdam, Dunkl operators, Bessel functions and the
  discriminant of a finite Coxeter group.  \textit{Compos. Math.} 85 (1993),
333--373.
\bibitem[O2]{O1} E.M. Opdam, Harmonic analysis for certain representations
of graded Hecke algebras. \textit{ Acta Math.}
 175 (1995), 75--112.
\bibitem[OV]{OV} A.L. Onishchik, E.B. Vinberg, Lie Groups and Algebraic Groups.
 Springer Verlag, Berlin, Heidelberg 1990. 
\bibitem[R1]{R1} M. R\"osler, Bessel convolutions on matrix cones. Compos. Math.
 143 (2007), 749--779.
\bibitem[R2]{R2} M. R\"osler, Positive convolution structure for a class of
Heckman-Opdam hypergeometric functions of type $BC$. J. Funct. Anal. 258
(2010), 
2779--2800.
\bibitem[RKV]{RKV}M. R\"osler, T. Koornwinder, M. Voit, Limit transition between
hypergeometric functions of type $BC$ and type $A$.
 Compos. Math. 149 (2013), 1381--1400.
\bibitem[RV1]{RV1} M. R\"osler, M.  Voit, Positivity of Dunkl's intertwining
operator via the 
trigonometric setting. \textit{IMRN} 63, 3379-3389 (2004). 
\bibitem[RV2]{RV2} M. R\"osler, M.  Voit,
A limit relation for Dunkl-Bessel functions of type A and B. 
SIGMA, Symmetry Integrability Geom. Methods Appl. 4, Paper 083, 9 pp. (2008).
\bibitem[RV3]{RV3} M. R\"osler, M.  Voit,
 Limit theorems for radial random walks on $p\times q$-matrices as $p$ tends to
infinity. 
\textit{Math. Nachr.} 284,  87-104 (2011). 
\bibitem[Sa]{Sa} P. Sawyer,  Spherical functions on ${\rm SO}\sb 0(p,q)/{\rm
SO}(p)\times{\rm SO}(q)$.  Canad. Math. Bull.  42  (1999),  486--498. 
\bibitem[Sch]{Sch} B. Schapira, Contributions to the hypergeometric function
theory of Heckman and Opdam: sharp estimates, Schwartz space, heat kernel,
Geom. Funct. Anal. 18  (2008), 222--250.
\bibitem[SK]{SK} J. Stokman, T. Koornwinder, Limit transitions for BC type
multivariable orthogonal polynomials. \textit{ Canad. J. Math.} 
49, 373--404 (1997).
\bibitem[Ti]{Ti} E.C. Titchmarsh, The theory of functions. Oxford Univ. Press,
  London, 1939.
\bibitem[V1]{V1}  M. Voit,
Limit theorems for radial random walks on homogeneous spaces with growing
dimensions. In:
J.~Hilgert et al. (eds.), Proc. symp. on infinite dimensional harmonic analysis
IV. 
Tokyo, World Scientific. 308-326 (2009). 
\bibitem[V2]{V2}  M. Voit, Central limit theorems for hyperbolic spaces and
  Jacobi processes on $[0,\infty[$. Monatsh. Math. 169,  441-468 (2013).
\bibitem[V3]{V3} M. Voit, Product formulas for a two-parameter family of
  Heckman-Opdam hypergeometric functions of type BC. Preprint 2013,
arXiv:1310.3075.
\bibitem[Zh]{Zh} F. Zhang, Quaternions and matrices of quaternions. 
\textit{Lin. Algebra Appl.}  251,  21-57 (1997).
\end{thebibliography}
\end{document}